\documentclass[11pt]{amsart}

\usepackage{mathtools, amssymb, amsthm}
\usepackage{xcolor}
\usepackage[all]{xy}
\usepackage{comment}
\usepackage{float}
\usepackage{graphicx} 
\usepackage{subcaption}
\usepackage{algorithm}
\usepackage{algpseudocode}
\usepackage{placeins}
\usepackage{tikz}
\usepackage{appendix}
\usepackage{tikz-3dplot}
\usepackage{tkz-euclide}
\usepackage[super]{nth}
\usepackage{nicematrix}
\usepackage{array}
\usepackage{arydshln}
\usepackage{enumitem}
\usepackage{pst-node}
\usepackage{tikz-cd} 
\usepackage{pifont}
\usepackage[T1]{fontenc}
\usepackage{bm}
\usepackage{relsize}
\usepackage{float}
\usepackage{placeins}
\usepackage{booktabs}
\usepackage{float}   
\usetikzlibrary{
  arrows.meta, bending,
  calc, positioning, fit, shapes,
  angles, quotes,
  automata,
  tikzmark
}

\newcounter{sdtri}

\newcounter{trilab}

\usepackage{listings}
\usepackage{hyperref}

\makeatletter
\def\bbordermatrix#1{\begingroup \m@th
  \@tempdima 4.75\p@
  \setbox\z@\vbox{%
    \def\cr{\crcr\noalign{\kern2\p@\global\let\cr\endline}}%
    \ialign{$##$\hfil\kern2\p@\kern\@tempdima&\thinspace\hfil$##$\hfil
      &&\quad\hfil$##$\hfil\crcr
      \omit\strut\hfil\crcr\noalign{\kern-\baselineskip}%
      #1\crcr\omit\strut\cr}}%
  \setbox\tw@\vbox{\unvcopy\z@\global\setbox\@ne\lastbox}%
  \setbox\tw@\hbox{\unhbox\@ne\unskip\global\setbox\@ne\lastbox}%
  \setbox\tw@\hbox{$\kern\wd\@ne\kern-\@tempdima\left[\kern-\wd\@ne
    \global\setbox\@ne\vbox{\box\@ne\kern2\p@}%
    \vcenter{\kern-\ht\@ne\unvbox\z@\kern-\baselineskip}\,\right]$}%
  \null\;\vbox{\kern\ht\@ne\box\tw@}\endgroup}
\makeatother
\usepackage{blkarray}

\makeatletter
  \renewcommand*\env@matrix[1][*\c@MaxMatrixCols c]{%
    \hskip -\arraycolsep
    \let\@ifnextchar\new@ifnextchar
  \array{#1}}
\makeatother

\newtheorem{theorem}{Theorem}[section]
\newtheorem{proposition}[theorem]{Proposition}
\newtheorem{corollary}[theorem]{Corollary}

\theoremstyle{definition}
\newtheorem{definition}[theorem]{Definition}
\newtheorem{example}[theorem]{Example}

\newtheorem{remark}[theorem]{Remark}

\newcommand{\bbZ}{\mathbb Z}

\newcommand{\bbF}{\mathbb{F}}
\newcommand{\low}{\mathrm{low}}
\newcommand{\dgm}{\mathrm{dgm}}
\newcommand{\bv}{\mathbf{v}}

\makeatletter
\@addtoreset{equation}{section}
\makeatother

\newdimen\numht
\newdimen\numwd

\definecolor{bubbles}{rgb}{0.91, 1.0, 1.0}
\definecolor{asparagus}{rgb}{0.53, 0.66, 0.42}
\definecolor{camel}{rgb}{0.76, 0.6, 0.42}
\definecolor{codegreen}{rgb}{0,0.6,0}
\definecolor{codegray}{rgb}{0.5,0.5,0.5}
\definecolor{codepurple}{rgb}{0.58,0,0.82}
\definecolor{backcolour}{rgb}{0.95,0.95,0.92}
\definecolor{airforceblue}{rgb}{0.36, 0.54, 0.66}

\newcommand\scalemath[2]{\scalebox{#1}{\mbox{\ensuremath{\displaystyle #2}}}}
\graphicspath{ {./images/} }

\newcommand*\circled[1]{\tikz[baseline=(char.base)]{
\node[shape=circle,draw,inner sep=2pt] (char) {#1};}}

\lstdefinestyle{mystyle}{
    backgroundcolor=\color{white},   
    commentstyle=\color{codegreen},
    keywordstyle=\color{magenta},
    numberstyle=\tiny\color{codegray},
    stringstyle=\color{codepurple},
    basicstyle=\ttfamily\footnotesize,
    breakatwhitespace=false,         
    breaklines=true,                 
    captionpos=b,                    
    keepspaces=true,                 
    numbers=left,                    
    numbersep=5pt,                  
    showspaces=false,                
    showstringspaces=false,
    showtabs=false,                  
    tabsize=1
}

\begin{document}

\title{Worst-Case Examples for the Computation of Persistent Homology}
\author{Uzay {\c C}etin and Erg{\" u}n Yal{\c c}{\i}n}
\date{\today}

\address{Uzay {\c C}etin,
Department of Mathematics, The Ohio State University, 
Columbus, OH, 43210-1174, USA}

\email{cetin.16@osu.edu}

\address{Erg\" un Yal\c c\i n, Department of Mathematics, Bilkent University, 06800 
Bilkent, Ankara, Turkiye}

\email{yalcine@fen.bilkent.edu.tr}

\begin{abstract}  We construct worst-case examples for the standard reduction algorithm for computing persistent homology. Our constructions are similar to the worst-case examples introduced by Morozov, but we replace the single-triangle arrangement with a strip of base and fin triangles. This structure allows us to give an explicit algorithm for their construction and to perform experiments comparing the runtime of different versions of the reduction algorithm.  
We further show that, after suitable edge and triangle subdivisions, these strip examples remain worst-case and can be realized as clique complexes of filtered graphs, and hence as Vietoris--Rips complexes of finite point clouds for a sequence of scale parameters. 
\end{abstract}
	

	
\maketitle

\section{Introduction}\label{sect:Introduction}

Due to the population increase and advances in internet and digital technologies, data has become one of the most important aspects of our lives. While classical mathematical methods such as statistical and clustering techniques are useful in data analysis, Topological Data Analysis (TDA) has emerged as an important new tool for analyzing data in last 20 years \cite{Carlsson-2009}. TDA is currently used in many areas, including genomics, robotics, medical imaging, cancer research, the spread of infectious diseases, financial networks, and natural language analysis (see, for example, \cite{ChazalMichel-2021}, \cite{DeyWang-2022}, \cite{RabadBlumberg-2020}, and \cite{CoskunuzerAkcora-2024}).  
  
Topological Data Analysis uses methods from algebraic topology to analyze data.  One of the main methods in TDA is persistent homology, which was developed by Edelsbruner, Letscher, Zomorodin, Carlsson, and others (see \cite{ELZ-2002}, \cite{ZomorodianCarlsson-2005}). In most of the applications, TDA via persistent homology takes as input a finite metric space embedded in $\mathbb{R}^m$ and produces persistent diagrams as output. One of the reasons that makes the persistent homology so effective is that it can be computed by a simple algorithm called the reduction algorithm. There are many computer packages based on different versions of the reduction algorithm that computes the persistent diagram of a filtered simplicial complex (see \cite{Otter-2017} for a comparison of the runtime performance of different packages). 
 
In \cite{Morozov-2005}, Morozov constructed a family of filtered simplicial complexes that provide worst-case examples for the reduction algorithm, exhibiting a runtime of $\Omega (N^3)$, where $N$ denotes the number of simplices. The main purpose of this paper is to give a new family of worst-case examples similar to Morozov's examples, but they are constructed by successively adding base and fin triangles to a strip of triangles instead of adding triangles inside a single triangle. We refer to these new set of worst-case examples as \emph{strip examples}. Because of their strip structure, 
it is easier to give an explicit algorithm for their construction and to perform experiments comparing the runtime of different versions of the reduction algorithm. 
 
In Theorem \ref{thm:Runtime}, we show that the runtime of the reduction algorithm on the strip worst-case examples is again $\Omega(N^3)$, where $N$ is the number of simplices. 
In Appendix \ref{sect:Algorithm}, we present a simple pseudo-code for constructing the strip examples for any number of triangles. In Section \ref{sect:Experiments} we use this code to perform computational experiments and report on the results comparing the runtimes of different versions of the reduction algorithm such as the standard reduction algorithm \cite{ELZ-2002}, the reduction algorithm with a twist \cite{ChenKerber-2011}, and the Look-ahead variant of the reduction algorithm \cite{MorozovSkraba-2024}.  

In Section \ref{sect:RealizingStrip}, we show that after modifying the strip examples by subdividing some of the triangles and edges, they can be realized as clique complexes of filtered graphs. Then we show that these filtered complexes can be realized as filtered graphs associated to a finite data cloud in some Euclidean space. As a consequence we obtain that these modified strip examples can be realized as Vietoris--Rips complexes of finite point clouds for some sequence of real numbers. 


\section{Matrix Reduction Algorithms}\label{sect:ReductionAlgorithms}

We refer the reader to \cite{DeyWang-2022}, \cite{RabadBlumberg-2020}, and \cite{CoskunuzerAkcora-2024} for basic definitions on persistent homology. Here we only give the definitions related to filtered simplicial complexes that will be used in the paper to explain the reduction algorithm.

\begin{definition}
A \emph{simplicial complex} $K$ consists of a pair $(V, S)$ where $V$ is the set of vertices of $K$, and $S$ is a set of subsets of $V$, called the simplices of $K$, satisfying the following properties:
\begin{enumerate}
\item for all $v\in V$, $\{ v \} \in S$,
\item if $\sigma \in S$ and $\tau \subseteq \sigma$, then $\tau \in S$.
\end{enumerate}
\end{definition}

A simplex $\sigma$ in $K$ is called a \emph{$k$-simplex} or a \emph{$k$-dimensional simplex} if $|\sigma | =k+1$. For a simplex $\sigma$ in $K$, a subset $\tau \subseteq \sigma$
is called a \emph{face} of $\sigma$. If $\dim \tau =\dim \sigma -1$, then $\tau$ is called a \emph{facet} of $\sigma$.
We say the simplicial complex $K$ is finite if its vertex set $V$ is finite. Throughout the paper we assume all the simplicial complexes are finite. The dimension of $K$ is the largest $d$ such that $K$ has a $d$-dimensional simplex.
We say that the simplicial complex $K'=(V', S')$ is a subcomplex of $K=(V, S)$ if $V'\subseteq V$ and $S' \subseteq S$. In this case we write $K' \subseteq K$.

\begin{definition} A nested sequence of simplical complexes
\[
\emptyset \subseteq K_1 \subseteq K_2 \subseteq \dots \subseteq K_m = K
\]
is called a \emph{filtration of $K$}. A simplicial complex with a filtration is called a \emph{filtered simplicial complex}. 
\end{definition}

A filtration function on a simplicial complex $K=(V, S)$ is a function $f : S \to  \mathbb{R}$ satisfying $f(\tau) \leq f ( \sigma )$ whenever $\tau \subseteq \sigma$. For each filtered complex $K$, we can define 
a filtration function $f: K \to \bbZ^{+}=\{1, 2, \dots\}$ such that for each simplex $\sigma$, $f(\sigma)$ is the smallest $i \in \bbZ$ such that $\sigma \in K_i$. Conversely, given a filtration function $f$ on a complex $K$, we can choose finitely many real numbers $r_1, \dots , r_m$ and define the filtration on $K$ by taking $K_i \subseteq K$ to be the subcomplex consists of simplices $\sigma\in K$ such that $f(\sigma) \leq r_i$.

Let $K=(V, S)$ be a filtered simplicial complex $\emptyset \subseteq K_1 \subseteq K_2 \subseteq \cdots \subseteq K_m=K$ with associated filtration function $f: K \to \mathbb{Z}^{+}$. Consider the chain complex $C_* (K)$ of $K$ in $\mathbb{F}_2$-coefficients. We can consider the chain complex $C_*(K)$ as a graded $\mathbb{F}_2$-vector space $C_* (K) = \bigoplus _{p=0} ^d C_p(K)$, where $d=\dim (K)$, and consider the family of boundary maps $\partial_*=\{ \partial_p : C_p (K)\to C_{p-1} (K)\}_{p=1} ^d$ as a linear transformation of degree -1 between graded vector spaces. 

To define a matrix representing the boundary map $\partial_* : C_* (K) \to C_*(K)$, we assume that the simplices in $K$ are totally ordered, i.e. sorted as a list $\sigma_1, \dots, \sigma_N$, in such a way that it satisfies the following properties:
\begin{enumerate}
\item If $\sigma_i$ is a face of $\sigma_j$, then $i<j$.
\item If $f(\sigma_i ) < f(\sigma_j)$, then $i<j$.
\end{enumerate}
Such an ordering of simplices of $K$ is called \emph{ordering compatible} with the filtration. Because of condition (1), the boundary matrix is strictly upper triangular.

\begin{definition}
Let $K$ be a filtered simplicial complex and 
$S=\{ \sigma_1,\sigma_2,\dots,\sigma_N\}$ be the set of simplices of $K$ ordered with an ordering compatible with the filtration. 
The \emph{boundary matrix} of $K$ is the $N \times N$-matrix $D$ with entries in $\mathbb{F}_2$ such that 
\[
D_{i,j}= \Biggl\{\begin{array}{cc}
1 & \text{if $\sigma_i$ is a facet of $\sigma_j$} \\
0 & \text{otherwise.}
\end{array}
\]
For each $j$, the $j$-th column of the boundary matrix $D$ is denoted by $D_j$.   
\end{definition}


\subsection{The Standard Algorithm for Reduction}

It is shown in \cite{ELZ-2002} that the computation of the persistent diagram can be accomplished by means of a simple
algorithm that brings the boundary matrix into a reduced form. 
To define this algorithm we need to introduce a function called the low function.

\begin{definition}\label{def:low}
For every nonzero vector $\bv$ with entries in $\bbF_2$,  the \emph{low of $\bv$} is defined by  
\[
\low (\bv) = \max (\{i \, |\, \bv [i] = 1\})
\]
where $\bv[i]$ denotes the $i$-th entry of $\bv$. We set $\low (\mathbf{0})=0$.
\end{definition}

The standard reduction algorithm is similar to the Gaussian elimination but it is performed on the columns. For each $j$ running from $1$ to $N$, 
the $j$-th column $D_j$ is modified as follows: If $\low(D_i) = \low(D_j)$ for some $i < j$, then the vector $D_j$ is replaced by $D_i + D_j$. Here the addition is performed in the field $\mathbb{F}_2$. Pseudocode for standard algorithm can be given in Algorithm \ref{alg:Standard}.

\begin{algorithm}
\caption{Standard Reduction Algorithm}\label{alg:cap}
\begin{algorithmic}
\State {\bf Input:} Boundary Matrix $D=[D_1|D_2|\cdots|D_N] \in \bbF_2 ^{N\times N}$
\State {\bf Output:} Reduced Matrix $R \in \bbF_2^{N\times N}$
\State $R \gets D;$
\For{$j=1$ to $N$} \Comment{$\Omega(N)$}
\While{$\exists i<j$ with $\low (D_i)=\low (D_j ) \neq 0$} \Comment{$\Omega(N)$}
\State $D_j \gets D_i+D_j$ \Comment{$\Omega(N)$}
\EndWhile
\EndFor
\end{algorithmic}
\label{alg:Standard}
\end{algorithm}

For every $k=1,\dots,d$, the persistent diagram $\dgm_k (K)$ for a filtered complex $K$ of dimension $d$ is the set of all persistent pairs $(f(\sigma_i), f(\sigma _j))$ such that $f(\sigma_i)$ is the filtration level where a cohomology class at dimension $k$ is born, and $f(\sigma_j)$ is the level where it dies (see \cite{DeyWang-2022} for details).
Note that if $R$ is the reduced matrix obtained from the boundary matrix of $K$ after applying the reduction algorithm, then the pair
$(f(\sigma _i), f(\sigma_j))$ is in the persistent diagram $\dgm _k(K)$ if and only if $\sigma_i$ is a $k$-dimensional simplex and $i = \low (R_j) \neq 0$. 

For all the columns of $R$ with $R_i=0$, we can conclude that there is a birth at $f(\sigma_i)$. If $R_i=0$ for some $i$ and there is no low at $R_{i,j}$ for any $j$, then we say there is a persistence pair $(f(\sigma_i), \infty)$.

\begin{example}\label{ex:Standard}  
Let $K$ be the filtered simplicial complex whose simplices in the order of appearance are given as follows: 
\[
K=\{ [1], [2], [3], [1,2], [1,3], [2,3], [1,2,3]\}.
\]
The boundary matrix of $K$ is:
\[
D=\bbordermatrix{ 
        & [1] & [2] & [3] & [1,2] & [1,3] & [2,3] & [1,2,3] \cr
    [1] &  0 & 0 & 0 & 1 & 1 & 0 & 0 \cr
    [2] &  0 & 0 & 0 & 1 & 0 & 1 & 0 \cr
    [3] &  0 & 0 & 0 & 0 & 1 & 1 & 0 \cr
    [1,2]  & 0 & 0 & 0 & 0 & 0 & 0 & 1 \cr
    [1,3]  & 0 & 0 & 0 & 0 & 0 & 0 & 1 \cr
    [2,3]  & 0 & 0 & 0 & 0 & 0 & 0 & 1 \cr
    [1,2,3]  & 0 & 0 & 0 & 0 & 0 & 0 & 0 \cr}
\]    
Then, the standard reduction algorithm is applied to the boundary matrix $D$ as follows:
\[
D =\scalemath{0.8}{\begin{bmatrix*}        
            0 & 0 & 0 & 1 & 1 & 0 & 0 \\
             0 & 0 & 0 & \textcolor{red}{1} & 0 & 1 & 0 \\
             0 & 0 & 0 & 0 & \textcolor{red}{1} & \textcolor{red}{1} & 0 \\
             0 & 0 & 0 & 0 & 0 & 0 & 1 \\
             0 & 0 & 0 & 0 & 0 & 0 & 1 \\
            0 & 0 & 0 & 0 & 0 & 0 & \textcolor{red}{1} \\
             0 & 0 & 0 & 0 & 0 & 0 & 0 
        \end{bmatrix*} 
        \xrightarrow{C5+C6 \xrightarrow{} C6} 
        \begin{bmatrix*}
             0 & 0 & 0 & 1 & 1 & 1 & 0 \\
             0 & 0 & 0 & \textcolor{red}{1} & 0 & \textcolor{red}{1} & 0 \\
             0 & 0 & 0 & 0 & \textcolor{red}{1} & 0 & 0 \\
             0 & 0 & 0 & 0 & 0 & 0 & 1 \\
             0 & 0 & 0 & 0 & 0 & 0 & 1 \\
             0 & 0 & 0 & 0 & 0 & 0 & \textcolor{red}{1} \\
             0 & 0 & 0 & 0 & 0 & 0 & 0 
        \end{bmatrix*} \xrightarrow{C4+C6 \xrightarrow{} C6} \begin{bmatrix*}
             0 & 0 & 0 & 1 & 1 & 0 & 0 \\
             0 & 0 & 0 & \textcolor{red}{1} & 0 & 0 & 0 \\
             0 & 0 & 0 & 0 & \textcolor{red}{1} & 0 & 0 \\
             0 & 0 & 0 & 0 & 0 & 0 & 1 \\
             0 & 0 & 0 & 0 & 0 & 0 & 1 \\
             0 & 0 & 0 & 0 & 0 & 0 & \textcolor{red}{1} \\
             0 & 0 & 0 & 0 & 0 & 0 & 0 
        \end{bmatrix*}}=R
\]        
\end{example}
If we take the filtration function to be the function that assigns $\sigma_i$ to the index $i$, then from the indices of the lows in  the reduction matrix, we can conclude that the set of persistent pairs for the filtered simplicial complex $K$ is
\[\dgm_0(K)=\{(1, \infty), (2, 4), (3, 5) \} \text{ and } \dgm_1(K)=\{(6,7)\}.
\]


\subsection{Reduction Algorithm with a Twist}

 Chen and Kerber \cite{ChenKerber-2011} introduced a variation of the reduction algorithm called the reduction algorithm with a twist.
 The idea behind this new algorithm is that if we do the reduction by first applying it to higher dimensional simplices, we will find a persistent pair $(f(\sigma_i), f(\sigma_j))$ after reducing the $j$-column before processing the $i$-th column. Since this pair tells us that the simplex $\sigma_i$ creates a homology class, we must have the $i$-th column in the reduced matrix equal to zero. So we can clear this column, i.e. set to zero without doing any column additions. This saves many unnecessary column operations and makes the reduction algorithm run much faster. 
 
To introduce the pseudocode for this algorithm, first note that the boundary matrix can be decomposed to smaller boundary matrices for linear maps $\partial _p : C_p (K) \to C_{p-1} (K)$ for $p=1, \dots, d$.   
The reduction algorithm can be described using these smaller boundary matrices. Pseudocode for the reduction algorithm with a twist is given in Algorithm \ref{alg:Twist}.

\begin{algorithm}[http]
\begin{algorithmic}
\State {\bf Input:} Boundary Matrix $D =[D_1|D_2|\cdots|D_N] \in \bbF_2 ^{N\times N}$
\State {\bf Output:} Reduced Matrix $R \in \bbF_2^{N\times N}$
\State $R \gets D;\hspace{1mm} L: [0,0,\dots,0]$
\For{$\delta =d,\dots,1$}
    \For{ $j=N,\dots,1$ }
        \If {$\dim \sigma _j =\delta$} 
        \While{$R_j \neq 0$ and $L[\low(R_j)] \neq 0$ }
        \State $R_j \gets R_j+R_{L[\low(R_j)]}$
        \EndWhile
        \If {$R_j \neq 0$}
        \State $i \gets \low(R_j)$
        \State $L[i] \gets j$
        \State $R_i \gets 0$
        \EndIf
        \EndIf
        \EndFor
        \EndFor
\end{algorithmic}
\caption{Reduction Algorithm with a Twist}
\label{alg:Twist}
\end{algorithm}


\begin{example}\label{ex:Twist} 
Let $D$ denote the boundary matrix of the filtered simplicial complex as in Example \ref{ex:Standard}. Then, the reduced matrix $R$ is obtained using Chen and Kerber's algorithm after applying only one step:

\[
D=\scalemath{0.8}{\begin{bmatrix*}
             0 & 0 & 0 & 1 & 1 & 0 & 0 \\
             0 & 0 & 0 & \textcolor{red}{1} & 0 & 1 & 0 \\
             0 & 0 & 0 & 0 & \textcolor{red}{1} & \textcolor{red}{1} & 0 \\
             0 & 0 & 0 & 0 & 0 & 0 & 1 \\
             0 & 0 & 0 & 0 & 0 & 0 & 1 \\
             0 & 0 & 0 & 0 & 0 & 0 & \textcolor{red}{1} \\
             0 & 0 & 0 & 0 & 0 & 0 & 0 
        \end{bmatrix*} \xrightarrow{\text{clear the $6$-th column}} \begin{bmatrix*}
            0 & 0 & 0 & 1 & 1 & 0 & 0 \\
            0 & 0 & 0 & \textcolor{red}{1} & 0 & 0 & 0 \\
            0 & 0 & 0 & 0 & \textcolor{red}{1} & 0 & 0 \\
            0 & 0 & 0 & 0 & 0 & 0 & 1 \\
            0 & 0 & 0 & 0 & 0 & 0 & 1 \\
            0 & 0 & 0 & 0 & 0 & 0 & \textcolor{red}{1} \\
            0 & 0 & 0 & 0 & 0 & 0 & 0 
        \end{bmatrix*}}=R
\]        
\end{example}

The twist algorithm is a very efficient algorithm, and it is used in all the current persistent homology calculation
packages such as Gudhi, Phat, and Ripser. We discuss the runtime of the twist algorithm in Section \ref{sect:Runtime}.


\subsection{Reduction Algorithm with Look-ahead} 

In a recent paper, Morozov and Skraba \cite{MorozovSkraba-2024} introduced an algorithm for persistent homology
in matrix multiplication time. In this paper they also discuss a look-ahead variant of the standard reduction algorithm. 
The main idea of this algorithm is that if the low of the column $R_j$ is $i\neq 0$, then for every column $R_{j'}$ with $j'>j$ and $R_{i, j'} =1$, we add column $R_j$ to $R_{j'}$ to make $R_{i, j'}=0$. Pseudocode for this algorithm is given in Algorithm \ref{alg:Lookahead}.

\begin{algorithm}[http]
\begin{algorithmic}
\State {\bf Input:} Boundary Matrix $D=[D_1|D_2|\cdots|D_N] \in \bbF_2 ^{N\times N}$
\State {\bf Output:} Reduced Matrix $R \in \bbF_2^{N\times N}$
\State $R \gets D;$
\For{$j=1$ to $N$} \Comment{$\Omega(N)$}
\If{$\low (R_j)\neq 0$}  
\State{$i \gets  \low(R_j)$}
\For{ $j'>j$ and $R[i, j']=1$}  \Comment{$\Omega(N)$}    
\State{ $R[\cdot, j']\gets R[\cdot, j'] + R[\cdot, j]$} \Comment{$\Omega(N)$}    
\EndFor
\EndIf
\EndFor
\end{algorithmic}
\caption{Reduction Algorithm with Look-ahead}\label{alg:Lookahead}
\end{algorithm}

\begin{example}\label{ex:Look-ahead} Let $D$ denote the boundary matrix as in Example \ref{ex:Standard}. Then, with the look-ahead reduction algorithm, the reduced matrix $R$ is obtained as follows:

\[
D=\scalemath{0.8}{\begin{bmatrix*}        
0 & 0 & 0 & 1 & 1 & 0 & 0 \\
0 & 0 & 0 & \textcolor{red}{1} & 0 & 1 & 0 \\
0 & 0 & 0 & 0 & \textcolor{red}{1} & \textcolor{red}{1} & 0 \\
0 & 0 & 0 & 0 & 0 & 0 & 1 \\
0 & 0 & 0 & 0 & 0 & 0 & 1 \\
0 & 0 & 0 & 0 & 0 & 0 & \textcolor{red}{1} \\
0 & 0 & 0 & 0 & 0 & 0 & 0 
\end{bmatrix*} 
\xrightarrow{C4+C6 \xrightarrow{} C6} 
\begin{bmatrix*}
0 & 0 & 0 & 1 & 1 & 1 & 0 \\
0 & 0 & 0 & \textcolor{red}{1} & 0 & 0 & 0 \\
0 & 0 & 0 & 0 & \textcolor{red}{1} & \textcolor{red}{1} & 0 \\
0 & 0 & 0 & 0 & 0 & 0 & 1 \\
0 & 0 & 0 & 0 & 0 & 0 & 1 \\
0 & 0 & 0 & 0 & 0 & 0 & \textcolor{red}{1} \\
0 & 0 & 0 & 0 & 0 & 0 & 0 
\end{bmatrix*} \xrightarrow{C5+C6 \xrightarrow{} C6} \begin{bmatrix*}
0 & 0 & 0 & 1 & 1 & 0 & 0 \\
0 & 0 & 0 & \textcolor{red}{1} & 0 & 0 & 0 \\
0 & 0 & 0 & 0 & \textcolor{red}{1} & 0 & 0 \\
0 & 0 & 0 & 0 & 0 & 0 & 1 \\
0 & 0 & 0 & 0 & 0 & 0 & 1 \\
0 & 0 & 0 & 0 & 0 & 0 & \textcolor{red}{1} \\
0 & 0 & 0 & 0 & 0 & 0 & 0 
\end{bmatrix*}}=R
\]
Note that in this example the reduced matrix is obtained using exactly the same number of operations as in the standard algorithm. This is not true in general. There are examples where the look-ahead  reduction algorithm performs better than the reduction algorithm. We discuss the runtime of the look-ahead variant of the reduction algorithm applied to the strip worst-case examples in Section \ref{sect:Experiments}. 
\end{example}


\subsection{The runtime of reduction algorithms}\label{sect:Runtime}

It is shown that for all three variants of the reduction algorithm, the runtime is $\Omega(N^3)$ in the worst-case where $N$ is the number of simplices. 

For the standard reduction algorithm, we see this by calculating runtime of the different steps of the algorithm.  
In this algorithm we have a \textbf{for} loop with $N$ elements so it runs in $\Omega(N)$ time. The \textbf{while} loop also runs in  $\Omega(N)$. Inside these loops the summation of two vectors of size $N$ is performed in $\Omega(N)$ time. So the runtime of the standard reduction algorithm is $\Omega(N^3).$

The runtime calculation for the reduction algorithm with a twist is similar. In this algorithm, the first two \textbf{for} loops run in $\Omega(N)$, the \textbf{while} loop runs in $\Omega(N)$, and the additions are performed in $\Omega(N)$ time.  So the complexity 
of the reduction algorithm with a twist is also $\Omega(N^3)$.

It is shown that in practice the twist algorithm is much more efficient because the boundary matrix is a sparse matrix in general. There is an explicit calculation for the runtime of this algorithm in terms of the number of persistent pairs in the reduced matrix (see  \cite{ChenKerber-2011}).
In the experiments that we performed with random data points of size 10, we observed that the number of vector additions that are performed goes down drastically by using the reduction algorithm with twist. Our experiments with the worst case
strip examples also support this claim (see Section \ref{sect:Experiments}).
 
In the look-ahead variant of the standard algorithm, we have two \textbf{for} loops each having runtime $\Omega(N)$, and there is a vector addition with runtime $\Omega(N)$. Hence the look-ahead algorithm also runs in $\Omega(N^3)$. Similar to the twist algorithm we can perform some calculations to compare the efficiency of the look-ahead reduction algorithm with other algorithms. In our experiments with a random data set of size 10, we obtained that the runtime of the look-ahead reduction algorithm is better than the runtime of the standard reduction algorithm.


\section{An Alternative Worst-case Example: Strip}\label{sect:StripExample}
    
In this section we construct a worst-case example for the reduction algorithm, called the strip example, similar to the worst-case examples introduced by Morozov in \cite{Morozov-2005}.

For each positive integer $n$, the \emph{strip worst-case complex} $X(n)$ is a filtered simplicial complex consisting of $n$ base triangles sitting on a plane and $n$ fin triangles with one edge lying on the plane and one vertex lying outside the plane.  We place these base triangles and fin triangles as shown in Figure \ref{fig:StripExample}. The base triangles are labeled $\circled{i}$ and the fin triangles are labeled  $\red \circled{-i}$ for $i=1, \dots, n$. The vertices of these triangles are labeled $v_1, \dots , v_{n+2}$ and $f_1, \dots, f_n$. The labeling of the base triangles starts from the center and then alternates left and right.

We label the horizontal edges of the base triangles  $1', \dots, n'$ such that $i'$ is the horizontal edge of the base triangle $\circled{i}$. The vertical base edges are labeled $1, \dots, n+1$ so that 1 is the edge between $v_1$ and $v_2$, and for all $i=2, \dots, n+1$, the edge $i$ is between $v_{i-1}$ and $v_{i+1}$ as shown in Figure \ref{fig:StripExample}. The base edge with label $n+1$ is also labeled as $0'$ to make it easier to justify its place in the filtration. 

The edges of the fin triangles are labeled $-1, \dots, -n$ and $-1^*, \dots, -n^*$. The edge $-1$ is between the vertices $f_1$
and $v_2$, and the edge $-1^*$ is between $f_1$ and $v_1$. For all $i=2, \dots, n$, the edge  $-i$ is between the vertices $f_i$ and $v_{i+1}$, and the edge $-i^*$ is between $f_i$ and $v_{i-1}$.
Note that there is one fin triangle above each vertical base edge except the edge with label $n+1=0'$. 
 

\begin{figure}[H]
\centering
\begin{tikzpicture}[
  scale=0.92,
  line cap=round, line join=round,
  every node/.style={font=\small},
  base/.style={draw=black, line width=0.5pt},
  blueE/.style={draw=blue, line width=0.5pt},
  fin/.style={draw=red, line width=0.5pt},
  dot/.style={circle, fill=black, inner sep=1pt},
  fdot/.style={circle, fill=red, inner sep=1pt}
]

\coordinate (L0) at (-6,0);
\coordinate (L1) at (-5,1.73);
\coordinate (L2) at (-4,0);
\coordinate (L3) at (-3,1.73);

\coordinate (A0) at (-2,0);
\coordinate (A1) at (-1,1.73);
\coordinate (A2) at (0,0);
\coordinate (A3) at (1,1.73);
\coordinate (A4) at (2,0);
\coordinate (A5) at (3,1.73);
\coordinate (A6) at (4,0);

\coordinate (R0) at (6,0);
\coordinate (R1) at (5,1.73);
\coordinate (R2) at (8,0);
\coordinate (R3) at (7,1.73);

\coordinate (fL) at (-6.2,4.8);   
\coordinate (f1) at (-0.7,4.8);   
\coordinate (fR) at (5.3,4.8);    

\draw[base] (L0)--(L2) (A0)--(A2)--(A4)--(A6) (R0)--(R2);
\draw[base] (L1)--(L3) (A1)--(A3)--(A5) (R1)--(R3);

\draw[base] (L0)--(L1) (L2)--(L1) (L2)--(L3);
\draw[base] (A0)--(A1) (A2)--(A1) (A2)--(A3)
           (A4)--(A3) (A4)--(A5) (A6)--(A5);
\draw[base] (R0)--(R1) (R2)--(R3) (R0)--(R3);

\draw[base] (L0)--(L2)--(L1)--cycle
  node[midway,xshift=10,yshift=-10] {$\circled{n-1}$};
\draw[base] (L1)--(L3)--(L2)--cycle
  node[midway,xshift=12,yshift=10] {$\circled{n-3}$};

\draw[base] (A4)--(A2)--(A3)--cycle
  node[midway,xshift=-12,yshift=-10] {$\circled 1$};
\draw[base] (A2)--(A3)--(A1)--cycle
  node[midway,xshift=13,yshift=10] {$\circled 2$};
\draw[base] (A4)--(A5)--(A3)--cycle
  node[midway,xshift=10,yshift=10] {$\circled 3$};
\draw[base] (A2)--(A0)--(A1)--cycle
  node[midway,xshift=-12,yshift=-10] {$\circled 4$};
\draw[base] (A6)--(A4)--(A5)--cycle
  node[midway,xshift=-12,yshift=-10] {$\circled 5$};

\draw[base] (R0)--(R2)--(R3)--cycle
  node[midway,xshift=10,yshift=-10] {$\circled{n}$};
\draw[base] (R1)--(R3)--(R0)--cycle
  node[midway,xshift=12,yshift=10] {$\circled{n-2}$};

\node at (-2.7,0.9) {$\dots$};
\node at (4.7,0.9) {$\dots$};
\node[red!75!black] at (-3.4,4.8) {$\dots$};
\node[red!75!black] at (2.6,4.8) {$\dots$};

\draw[blueE] (A3)--(A5) node[midway,above,scale=0.8] {$3'$};
\draw[blueE] (A1)--(A3) node[midway,above,scale=0.8] {$2'$};
\draw[blueE] (R1)--(R3) node[midway,above,scale=0.6] {$(n-2)'$};
\draw[blueE] (L1)--(L3) node[midway,above,scale=0.6] {$(n-3)'$};

\draw[blueE] (A2)--(A4) node[midway,below,scale=0.8] {$1'$};
\draw[blueE] (A0)--(A2) node[midway,below,scale=0.8] {$4'$};
\draw[blueE] (A4)--(A6) node[midway,below,scale=0.8] {$5'$};
\draw[blueE] (R0)--(R2) node[midway,below,scale=0.7] {$n'$};
\draw[blueE] (L0)--(L2) node[midway,below,scale=0.6] {$(n-1)'$};

\draw[blueE] (R2)--(R3) node[midway,right,scale=0.7] {$n+1=0'$};

\draw[base] (A2)--(A3) node[midway,left,scale=0.65] {$1$};
\draw[base] (A4)--(A3) node[midway,left,scale=0.65] {$2$};
\draw[base] (A2)--(A1) node[midway,left,scale=0.65] {$3$};
\draw[base] (A4)--(A5) node[midway,left,scale=0.65] {$4$};
\draw[base] (A0)--(A1) node[midway,left,scale=0.65] {$5$};
\draw[base] (A5)--(A6) node[midway,right,scale=0.65] {$6$};

\node[left,scale=0.7] at (-5.5,0.866) {$n$};
\node[right, xshift=2, yshift=5, scale=0.6] at (6.5,0.866) {$n-1$};

\draw[red, fin] (L0)-- node[midway,yshift=18,xshift=-13] {$-n$}
           (fL) -- node[midway,yshift=5,xshift=10] {$-n^{*}$} (L1);
\node[red,scale=0.8] at (-5.7,2.5) {$\circled{-n}$};
\node[fdot] at (fL) {};
\node[left,red] at ($(fL)+(-0.1,0.1)$) {$f_n$};

\draw[red, fin] (A2)-- node[midway,yshift=18,xshift=-14] {$-1$}
           (f1) -- node[midway,yshift=5,xshift=10] {$-1^{*}$} (A3);
\node[red,scale=0.8] at (0.1,2.5) {$\circled{-1}$};
\node[fdot] at (f1) {};
\node[left,red] at ($(f1)+(-0.1,0.1)$) {$f_1$};

\draw[red, fin] (R0)-- node[midway,yshift=18,xshift=-22,scale=0.7] {$(-n+1)^{*}$}
           (fR) -- node[midway,yshift=5,xshift=13,scale=0.7] {$-n+1$} (R3);
\node[red,scale=0.6] at (6.1,2.5) {$\circled{-n+1}$};
\node[fdot] at (fR) {};
\node[left,red] at ($(fR)+(-0.1,0.1)$) {$f_{n-1}$};

\foreach \P in {L0,L1,L2,L3,A0,A1,A2,A3,A4,A5,A6,R0,R1,R2,R3}{
\node[dot] at (\P) {};
}

\node[below] at (L0) {$v_{n+1}$};
\node[above] at (L3) {$v_{n-5}$};
\node[below] at (L2) {$v_{n-3}$};
\node[above] at (L1) {$v_{n-1}$};

\node[below] at (A0) {$v_6$};
\node[above] at (A1) {$v_4$};
\node[below] at (A2) {$v_2$};
\node[above] at (A3) {$v_1$};
\node[below] at (A4) {$v_3$};
\node[above] at (A5) {$v_5$};
\node[below] at (A6) {$v_7$};

\node[below] at (R0) {$v_{n-2}$};
\node[above] at (R1) {$v_{n-4}$};
\node[above] at (R3) {$v_n$};
\node[below] at (R2) {$v_{n+2}$};

\end{tikzpicture}
\caption{The simplicial complex \(X(n)\) for \(n\equiv 1 \pmod 4\).}
\label{fig:StripExample}
\end{figure}
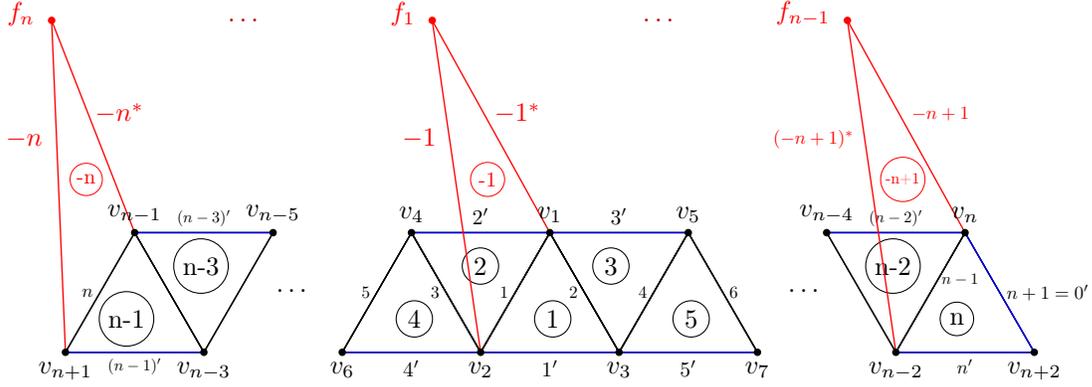

The strip worst-case examples grow inductively by adding triangles. To reach from $X(n)$ to $X(n+1)$ we just add a fin triangle above the edge $n+1$ and a base triangle next to the edge with label $n$. Note that the simplicial complex $X(n)$ has $2n+2$ vertices, $4n+1$ edges, and $2n$ triangles. The number of simplices of $X(n)$ is $N=8n+3$. The Euler characteristic of $X(n)$ is equal to 1 since it has homology isomorphic to the homology of a point. To define the filtration on $X(n)$, we order the simplices of the complex using the following ordering: 
\begin{enumerate}
\item Fin in vertices $f_1, \dots, f_n$,
\item Base vertices $v_1, \dots, v_{n+2}$, 
\item Horizontal base edges that destroy components $n',(n-1)',\dots,1'$,
\item Base edge $n+1$ which is also labeled as $0'$,
\item Fin edges that destroy fin components $-n^{*},\dots,-1^{*}$, 
\item Fin edges that creates cycles $-n,-n+1,\dots,-1$, 
\item Vertical base edges $n, n-1, \dots, 1$, 
\item Base triangles $\circled{1},\circled{2},\dots,\circled{n}$, and
\item Fin triangles $\red \circled{-1}, \red \circled{-2}, \dots, \red\circled{-n}$.
\end{enumerate}

To simplify calculations, in the proof of Theorem \ref{thm:Runtime} we apply the reduction algorithm only on the boundary matrix for $\partial _2 : C_2 (X(n)) \to C_1 (X(n)).$ We denote this matrix by $B(n)$ and the associated reduced matrix by $R(n)$.

\begin{example}\label{ex:B(2)}
For $n=2$, the boundary matrix $B(2)$ and the reduced matrix $R(2)$ are as follows:
\[
B(2) =\bbordermatrix{ 
& \circled{1} & \circled{2} & \red \circled{-1} & \red \circled{-2} \cr
\ 2' &  0 & 1 & 0& 0   \cr
\ 1' &  1 & 0 & 0 & 0   \cr
\ 0' &  0 & 1 & 0& 0  \cr
-2^* & 0 & 0 & 0 & 1 \cr
-1^* & 0 & 0 & 1 & 0 \cr 
-2 & 0 & 0 & 0& 1   \cr
-1 & 0 & 0 & 1& 0   \cr
\ 2 & 1 & 0 & 0& 1  \cr
\ 1 & 1 & 1 & 1 & 0  \cr}
 \xrightarrow{\substack{C_1 +C_2 \rightarrow C_2\\
                         C_1 +C_3 \rightarrow C_3\\
                         C_2+C_3 \rightarrow C_3 \\
                         C_2+C_4 \rightarrow C_4}} 
\bbordermatrix{ 
& \circled{1} & \circled{2} & \red \circled{-1} & \red \circled{-2} \cr
&  0 & 1 & 1 & 1  \cr
&  1 & 1 & 0& 1\cr
&  0 & 1 & 1 & 1 \cr
& 0 & 0 & 0& 1 \cr
& 0 & 0 & 1 & 0  \cr
& 0 & 0 & 0& 1 \cr
& 0 & 0 & 1 & 0  \cr  
& 1 & 1 & 0 & 0  \cr
& 1 & 0 & 0 & 0 \cr}=R(2)
\]
The processing of the columns of the matrix $B(2)$ can be explained using the sparse matrix notation:
\begin{align*}
\circled{1} & \to (1, 2, 1') 
\end{align*}
Note that the low of the column for $\circled{1}$ is the row corresponding to the edge $1$. In this case we say the triangle $\circled{1}$ is paired with the edge $1$. For the triangle $\circled{2}$, we have 
\begin{align*} 
\circled{2} & \to (1, 0', 2') \\
\circled{1}+\circled{2} & \to  (1, 2,1')+ (1, 0', 2')
=(2, 0', 1', 2')
\end{align*}
This pairs the triangle $\circled{2}$ with the edge $2$. For the fin triangles we have a similar pairing process.  
\begin{align*}
{\red \circled{-1}} & \to (1, -1, -1^*)  \\
{\red \circled{-1}} + \circled{1} & \to (1, -1, -1^*)+ (1, 2,1')=(2, -1, -1^*, 1') \\
{\red \circled{-1}}+ \circled{1} + (\circled{1} + \circled{2}) & \to (2, -1, -1^*, 1')+ (2, 0', 1', 2')=  (-1,-1^*, 0', 2') \end{align*}
This makes the fin triangle $\red \circled{-1}$ pair with the edge $-1$. Finally for the fin triangle $\red \circled{-2}$, we have
\begin{align*}
{\red \circled{-2}} & \to (2, -2, -2^*)  \\
{\red \circled{-2}} + ( \circled{1} + \circled{2} ) & \to (2, -2, -2^*)+ (2, 0', 1', 2')= (-2, -2^*, 0', 1', 2') 
\end{align*}
This pairs the triangle $\red \circled{-2}$ with the edge $-2$.
So the persistent pairs for the 1-dimensional homology $H_1$ are
\[
\dgm_1 (X(2)) =\{ (1, \circled{1}), (2, \circled{2}), (-1, {\red \circled{-1}}), (-2, {\red \circled{-2}} ) \}.
\]
Here we listed the persistent pairs using the simplices since in the above matrices we used the simplices for labeling the columns and rows. 
Note that these pairs correspond to the indices of the lows in the reduced matrix $R(2)$.   
\end{example}
    
Now we prove our main result in this section.

\begin{theorem}\label{thm:Runtime}
Let $B(n)$ denote the matrix for the boundary map $\partial_2: C_2 (X(n))\to C_1 (X(n))$.   The runtime on the standard reduction algorithm applied to $B(n)$ has complexity $\Omega(n^3)$. As a consequence, the complexity of the standard reduction algorithm for the boundary matrix of $X(n)$ is $\Omega(N^3)$ where $N=8n+3$ is the number of simplices of $X(n)$. 
\end{theorem}
    
\begin{proof} Extending the pairing algorithm given above for $n=2$, we see that the base triangles $\circled {i}$ will be paired by the base edges $i$ with the following list of representatives:
\begin{align*} \circled{1} & \to (1, 2, 1') \\    
\circled{1}+\circled{2} & \to (1, 2, 1')+(1,3, 2')= (2, 3, 1', 2') \\
\circled{1}+\circled{2} +\circled{3} & \to  (2,3, 1', 2')+(2,4,3')=(3, 4, 1', 2', 3') \\
& \vdots \\
\circled{1}+\circled{2} +\cdots + \circled{n-1} & \to (n-1, n, 1', 2', \cdots, (n-1)')  \\
\circled{1}+\circled{2} +\cdots + \circled{n} & \to (n, 0', 1', 2', \cdots, n') 
\end{align*}   
This table should be read in the following way. The boundary of $\circled{1}$ is $(1, 2, 1')$ as an ordered 3-tuple of edges. So $\circled{1}$ is paired with $1$. 
The boundary of the triangle $\circled{2}$ is  $(1, 3, 2')$. 
Since $1$ is already paired with $\circled{1}$, we need to add the column of $\circled{1}$ to the column of $\circled{2}$. This creates a chain $\circled{1}+\circled{2}$ with boundary $(2,3,1', 2')$. So the triangle $\circled{2}$ is paired with $2$. We continue this way until for all $i$, the triangle \circled{i} is paired with edge $i$. From the table we see that in the reduced matrix $R(n)$, the column for \circled{i} has $i+2$ nonzero entries.
For each $i$, let us denote the chain $\circled{1}+\cdots + \circled{i}=(i, i+1, 1', 2', \dots, i')$ by $c_i$. 

The complexity of processing the base triangles can be calculated as follows: Since the complexity of adding two vectors with $a_1$ and $a_2$ nonzero entries is equal to $a_1+a_2$, the complexity of performing all the above additions is
\[
(3+3)+(4+3)+(5+3)+\cdots+((n+1)+3)=\frac{1}{2}(n^2+9n-10).
\]

Now we calculate the complexity of processing the fin triangles. For each $i$, the processing of the fin triangle $\red{\circled{-1}}$ is illustrated in the following table:
\begin{align*}
{\red{\circled{-1}}} & \to (1, -1, -1^*) \\
{\red{\circled{-1}}} + c_1 & \to (1, -1, -1^*) +(1,2,1')=(2, -1, -1^*, 1') \\
\bigl( {\red{\circled {-1}}}+ c_1 \bigr )+ c_2 & \to (2, -1, -1^*, 1')+ (2,3,1', 2')=(3, -1, -1^*, 2') \\
\bigl( {\red{\circled{-1}}} +c_1+c_2 \bigr)  +c_3 & \to (3, -1, -1^*, 2')+ (3,4,1', 2', 3')= (4, -1, -1^*, 1', 3') \\
\bigl ( {\red{\circled{-1}}}+c_1 + c_2 + c_3 \bigr)  + c_4 & \to (4,-1, -1^*, 1', 3') +(4, 5, 1', 2', 3', 4') = (5, -1, -1^*, 2', 4') \\
& \vdots \\
\bigl ( {\red{\circled{-1}}}+c_1 + c_2 + \cdots + c_{n-1} \bigr ) +c_n & 
\to \begin{cases}  (-1, -1^*, 0', 1', 3', 5', \cdots, n' ) & \text{ if } n \text{ is odd} \\
(-1, -1^*, 0', 2', 4', 6', \dots, n') & \text{ if } n \text{ is even.} 
\end{cases}
\end{align*} 
This makes the fin triangle $\red{\circled{-1}}$ pair with the edge $-1$. There is a similar process for each of the fin triangles. If we assume that the processing of the fin triangle $\red{\circled{-1}}$ is done by performing the additions ${\red \circled{-1}} +c_1+\cdots+ c_n$, then the total number of field additions on nonzero entries is equal to 
\[
3+\sum_{i=1} ^n (i+2)=\frac{1}{2} (n^2 +5n+6).
\]
In a similar way if we process the $j$-th fin triangle $\red{\circled{-j}}$, we see that we will need to perform the addition ${\red \circled{-j}} + c_j +\dots+ c_n$, hence the total number of additions is
\[
3+\sum_{i=j} ^n (i+2)=\frac{1}{2} (n^2 +5n + -j^2-3j +10).
\]
Summing these over all $j=1, \dots, n$, we obtain that the number of additions is equal to 
\[
\frac{1}{6}n (2n^2 +9n +25).
\]
Adding this to the complexity of processing the base triangles, we find the total number of additions is
\[
S(n)=\frac{1}{3} (n^3+6n^2+26n-15).
\]
In particular, $S(n)=\Omega (n^3)$. Since the number of simplices in the complex is $N=8n+3$, it follows that the worst-case running time of the standard reduction algorithm is $\Omega(N^3)$.
\end{proof}

\begin{remark}
 In our implementation, additions are carried out at each intermediate step, so the actual number of additions performed is usually larger than $T(n)$.   
\end{remark}


\section{Experiments with the strip example}\label{sect:Experiments}

One of our motivations for introducing an alternative worst-case example is to be able to write computer programs that use worst-case examples for performance tests. The linear structure of the strip example allows us to give  a simple explicit algorithm for producing the complex $X(n)$ for every $n\geq 1$. This algorithm can be found in Appendix \ref{sect:Algorithm}. The Python code written using this algorithm and the results of the experiments that we run can be found on the github page: 
\href{https://github.com/Ergun1234/Morozov_type_Examples}{https://github.com/Ergun1234/Morozov\_type\_Examples}.

After running the Python code for constructing the strip example $X(n)$ and running the standard reduction algorithm for the boundary matrix of $X(n)$, we obtained the persistent diagrams for the strip examples.  Figure~\ref{fig:SplitSimplicialComplexes} shows the filtered simplicial complexes $X(n)$ for some values of $n$, and Figure \ref{fig:persistence-diagrams} illustrates the corresponding persistence diagrams.

\begin{figure}[htbp]
  \centering
  \begin{subfigure}{0.32\textwidth}
    \centering
    \includegraphics[width=\textwidth]{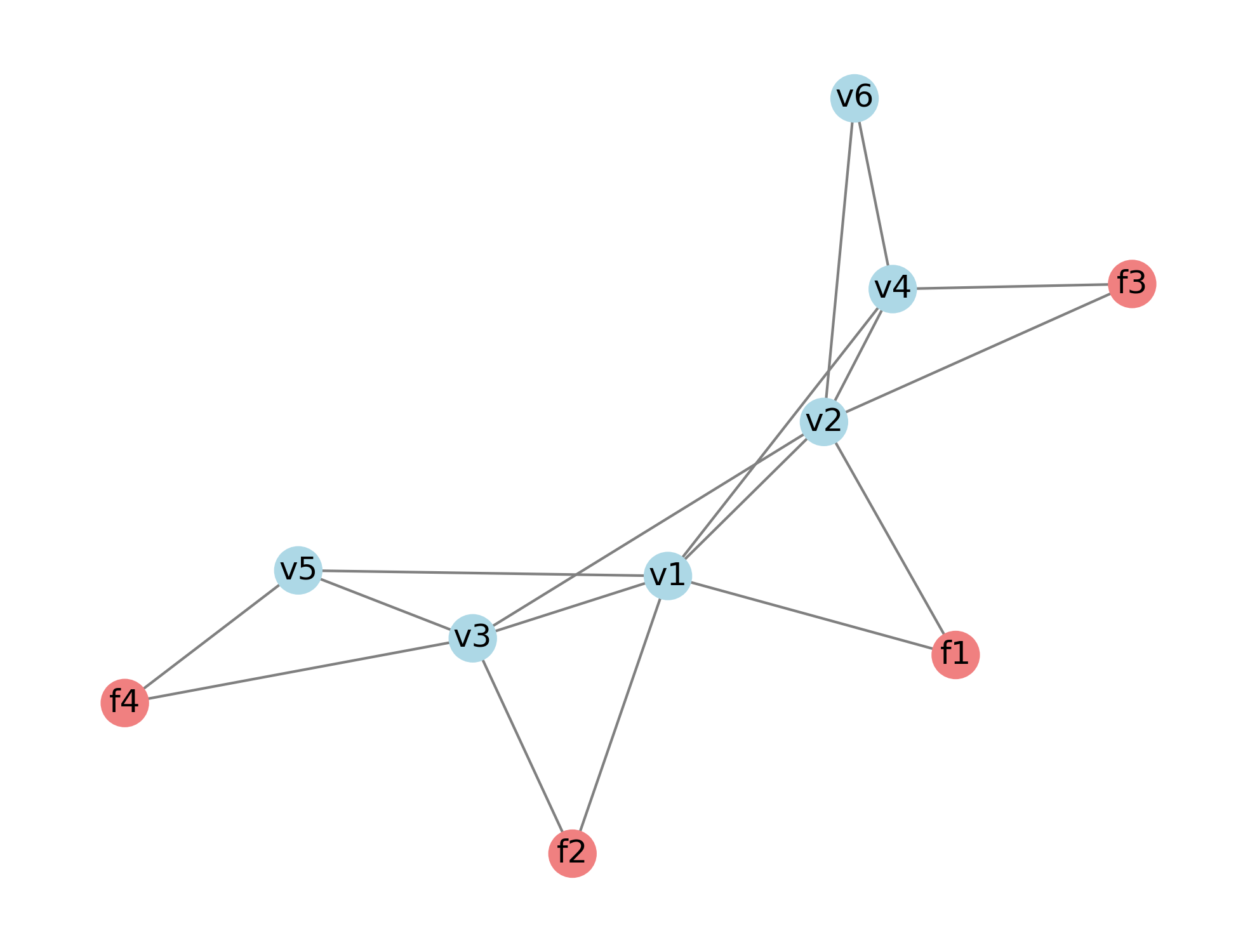}
    \caption{$n=4$}
  \end{subfigure}
  \hfill
  \begin{subfigure}{0.32\textwidth}
    \centering
    \includegraphics[width=\textwidth]{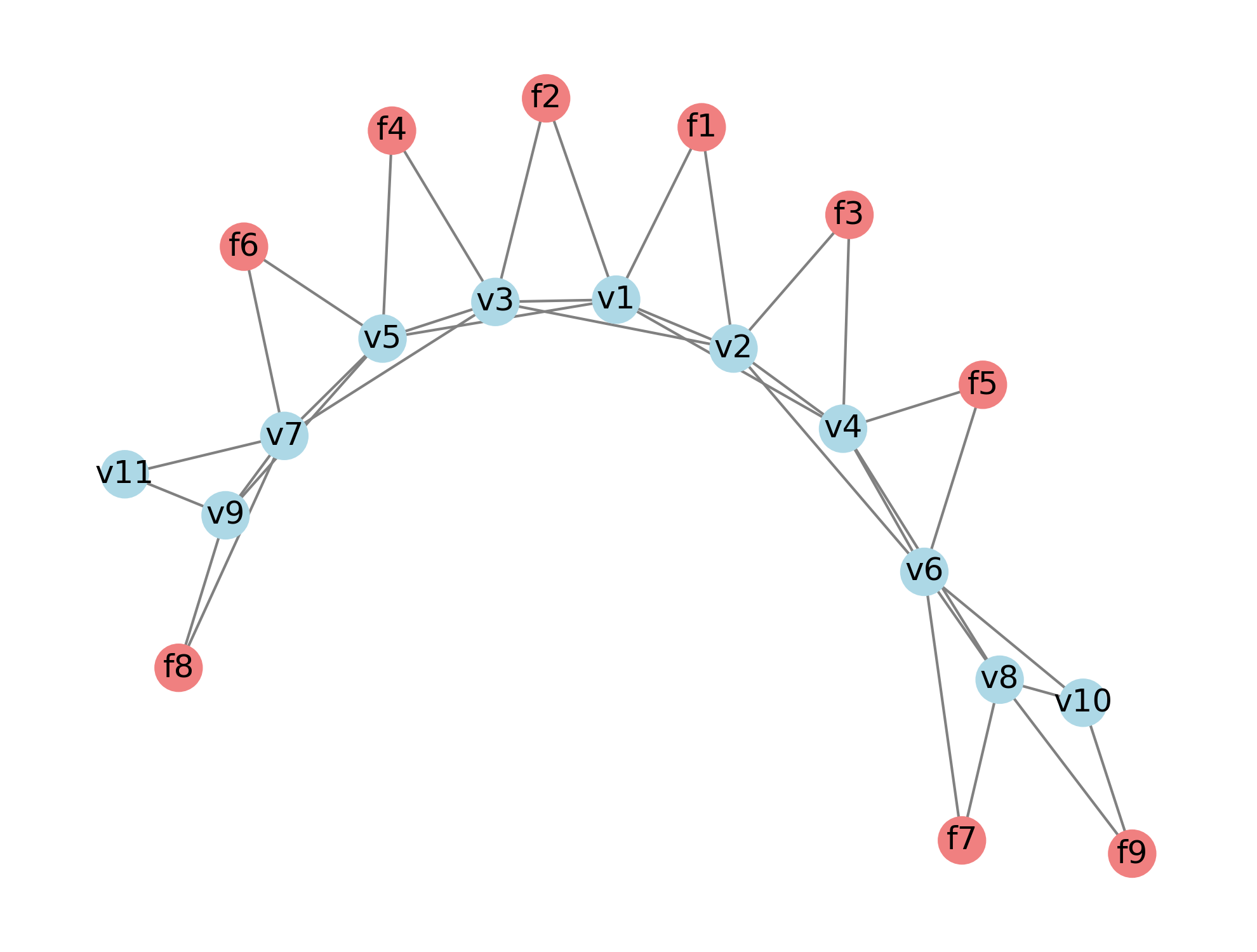}
    \caption{$n=9$}
  \end{subfigure}
  \hfill
  \begin{subfigure}{0.32\textwidth}
    \centering
    \includegraphics[width=\textwidth]{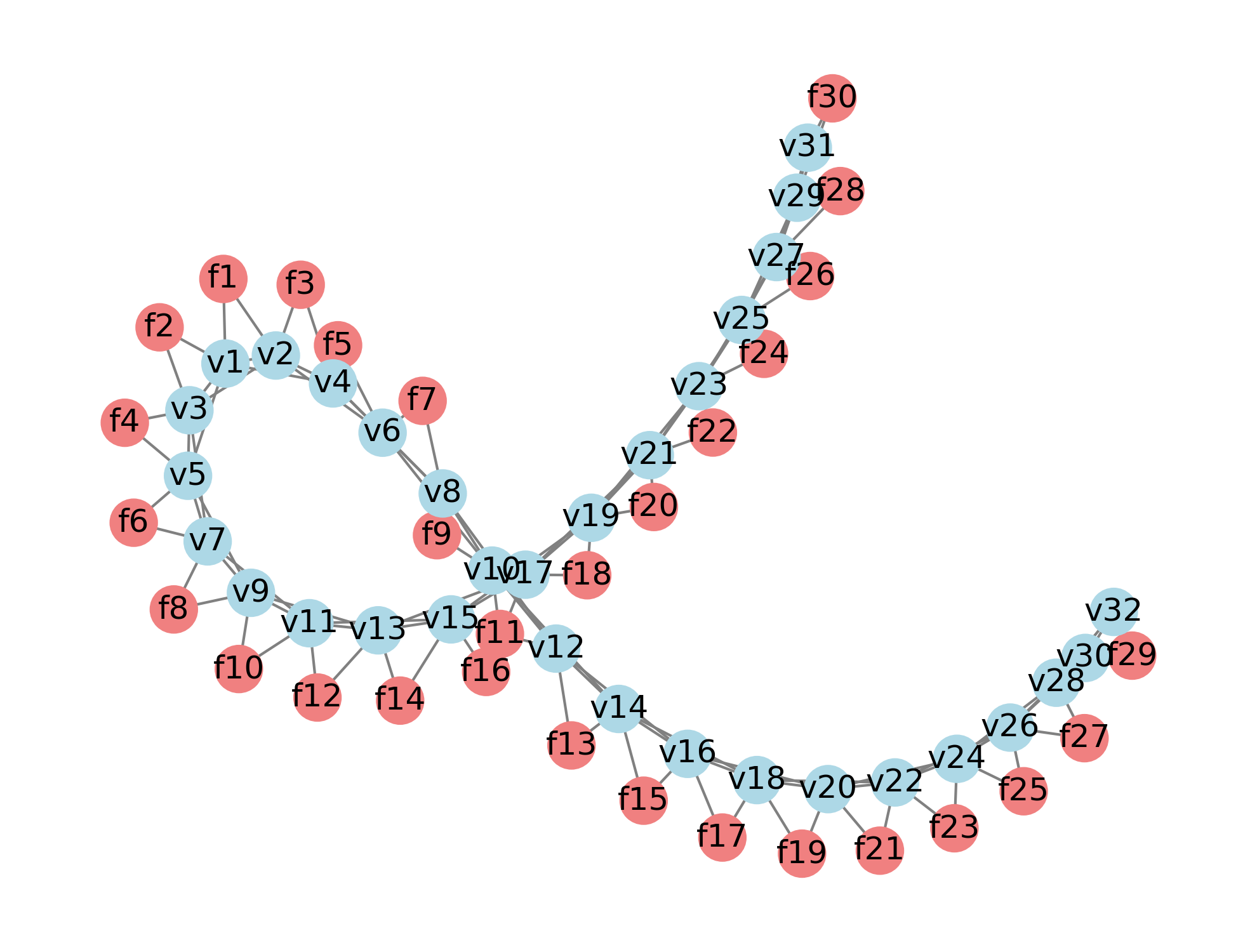}
    \caption{$n=30$}
  \end{subfigure}
  \caption{The $1$-skeletons of the strip worst-case simplicial complexes $X(n)$ for different values of $n$.}
  \label{fig:SplitSimplicialComplexes}
\end{figure}

\begin{figure}[htbp]
  \centering
  \begin{subfigure}{0.32\textwidth}
    \centering
    \includegraphics[width=\textwidth]{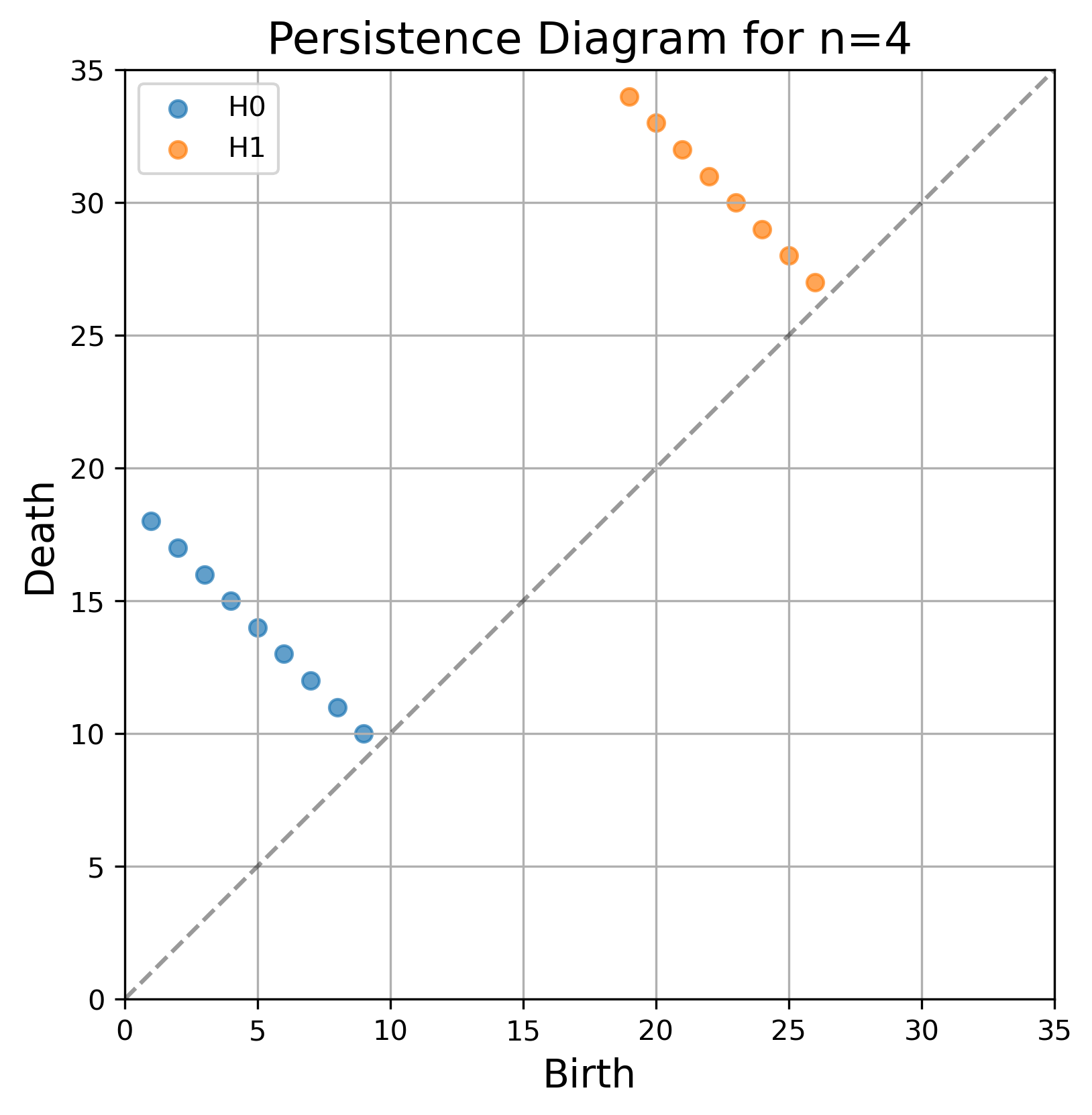}
    \caption{$n=4$}
    \label{fig:pd-n4}
  \end{subfigure}
  \hfill
  \begin{subfigure}{0.32\textwidth}
    \centering
    \includegraphics[width=\textwidth]{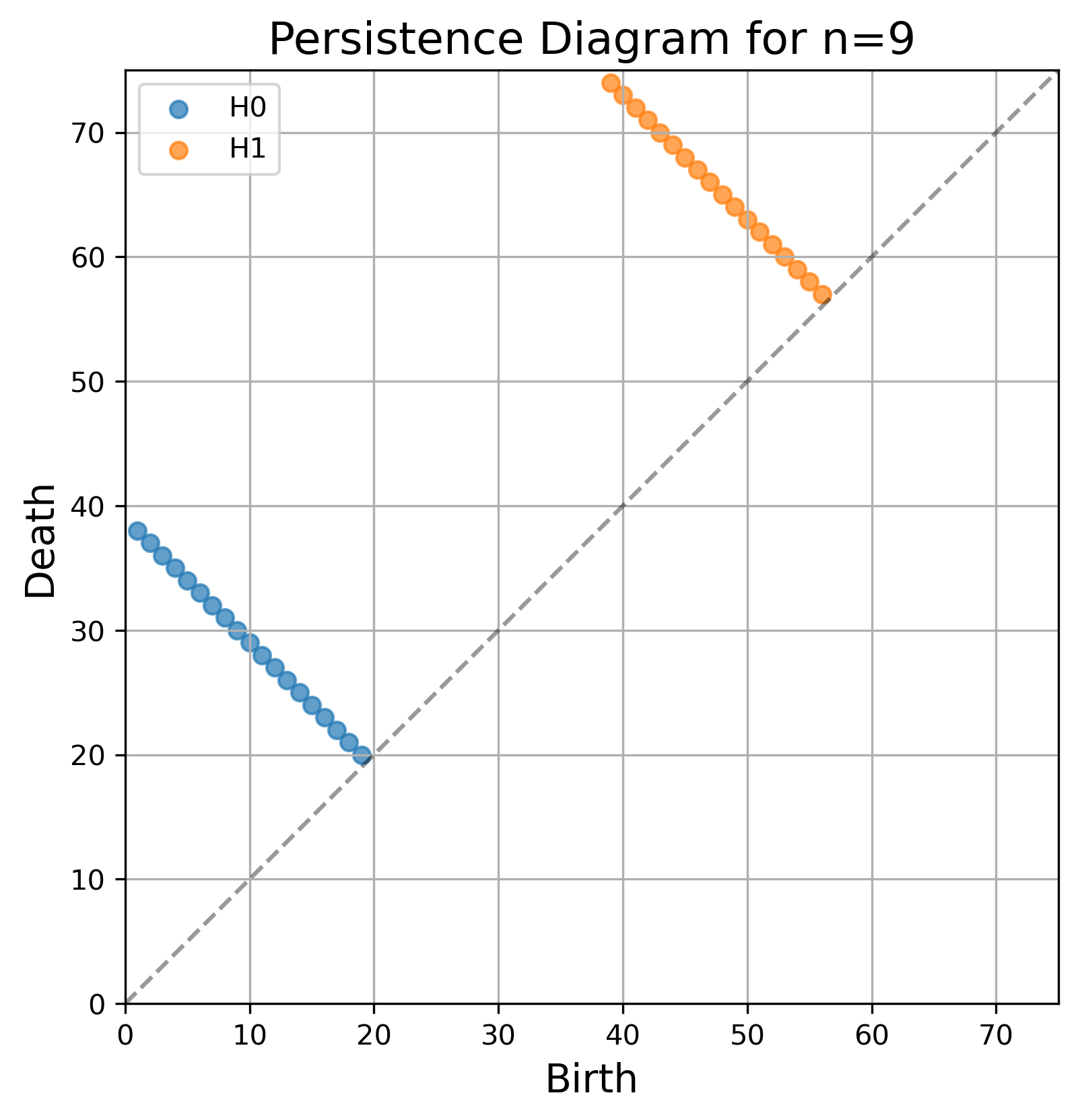}
    \caption{$n=9$}
    \label{fig:pd-n9}
  \end{subfigure}
  \hfill
  \begin{subfigure}{0.32\textwidth}
    \centering
    \includegraphics[width=\textwidth]{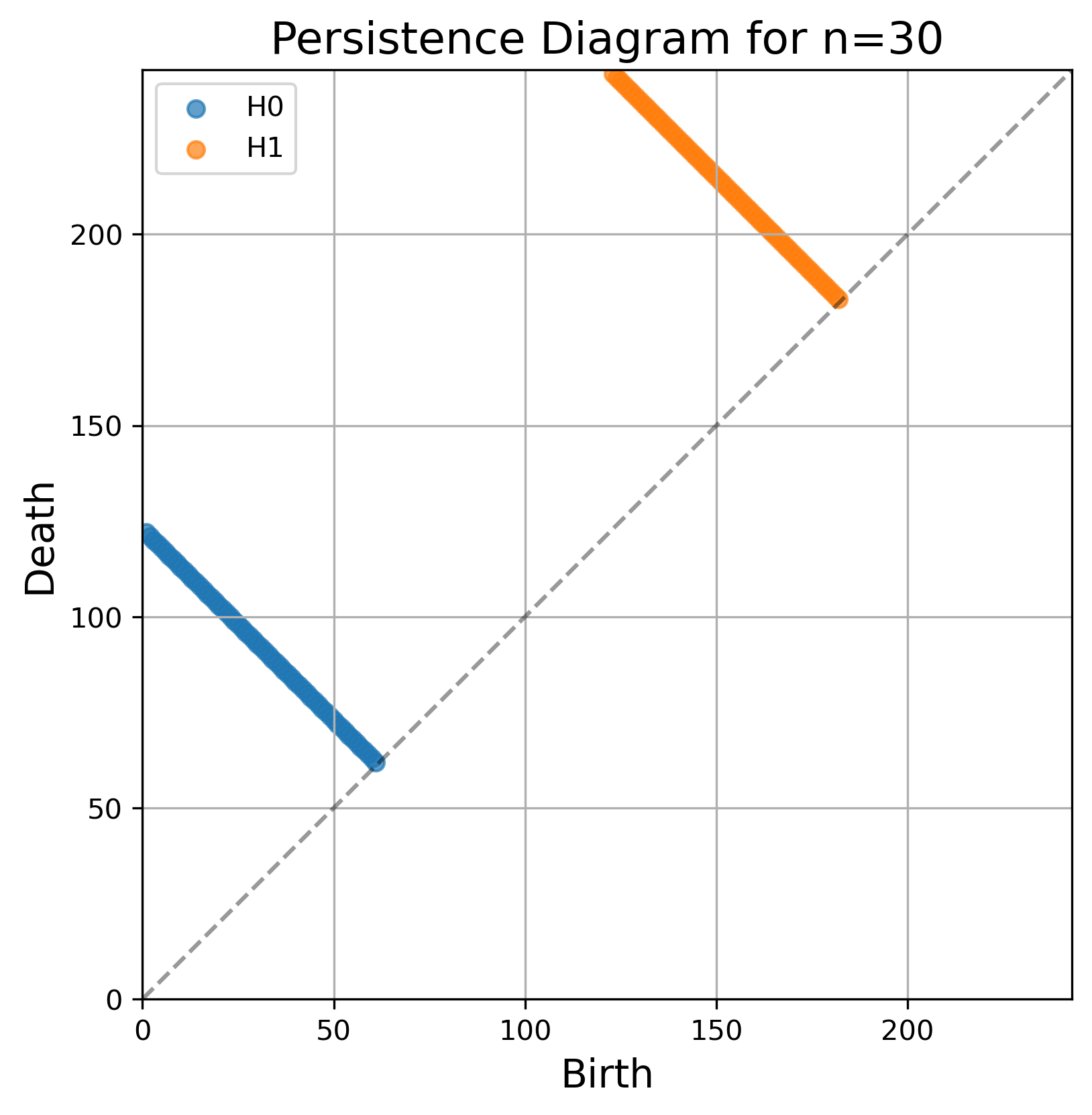}
    \caption{$n=30$}
    \label{fig:pd-n30}
  \end{subfigure}
  \caption{Persistence diagrams for strip worst-case examples $X(n)$ for different values of $n$. 
  Blue points correspond to $H_0$ features and orange points correspond to $H_1$ features.
  The dashed line indicates the diagonal $y=x$.}
  \label{fig:persistence-diagrams}
\end{figure}
 
We performed further experiments with the split worst-case examples and calculated the runtime of the reduction algorithms, the standard reduction algorithm, the reduction algorithm with a twist, and the reduction algorithm with look-ahead. Figure \ref{fig:runtime-two-panels} shows the running times of these three algorithms where the runtime is calculated using a counter that tracks the number of additions performed during the calculation of the reduced boundary matrix. Figure~\ref{fig:elapsed-two-panels} shows the elapsed running times of the algorithms on a MacBook with an M5 processor.
\begin{figure}[htbp]
  \centering
  \begin{subfigure}{0.48\textwidth}
    \centering
    \includegraphics[width=\textwidth]{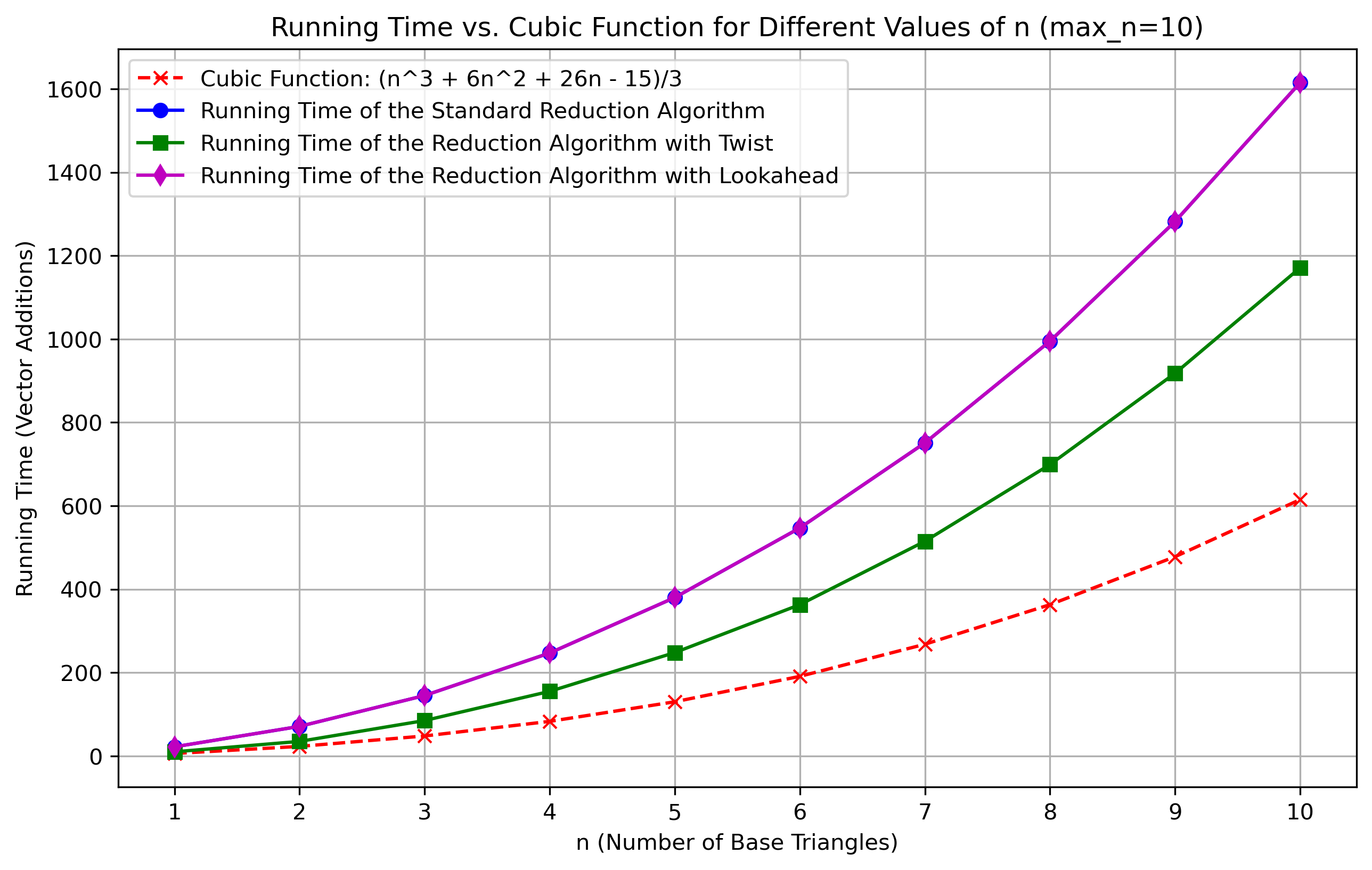}
    \caption{$\max_n=10$}
    \label{fig:runtime-n10}
  \end{subfigure}
  \hfill
  \begin{subfigure}{0.48\textwidth}
    \centering
    \includegraphics[width=\textwidth]{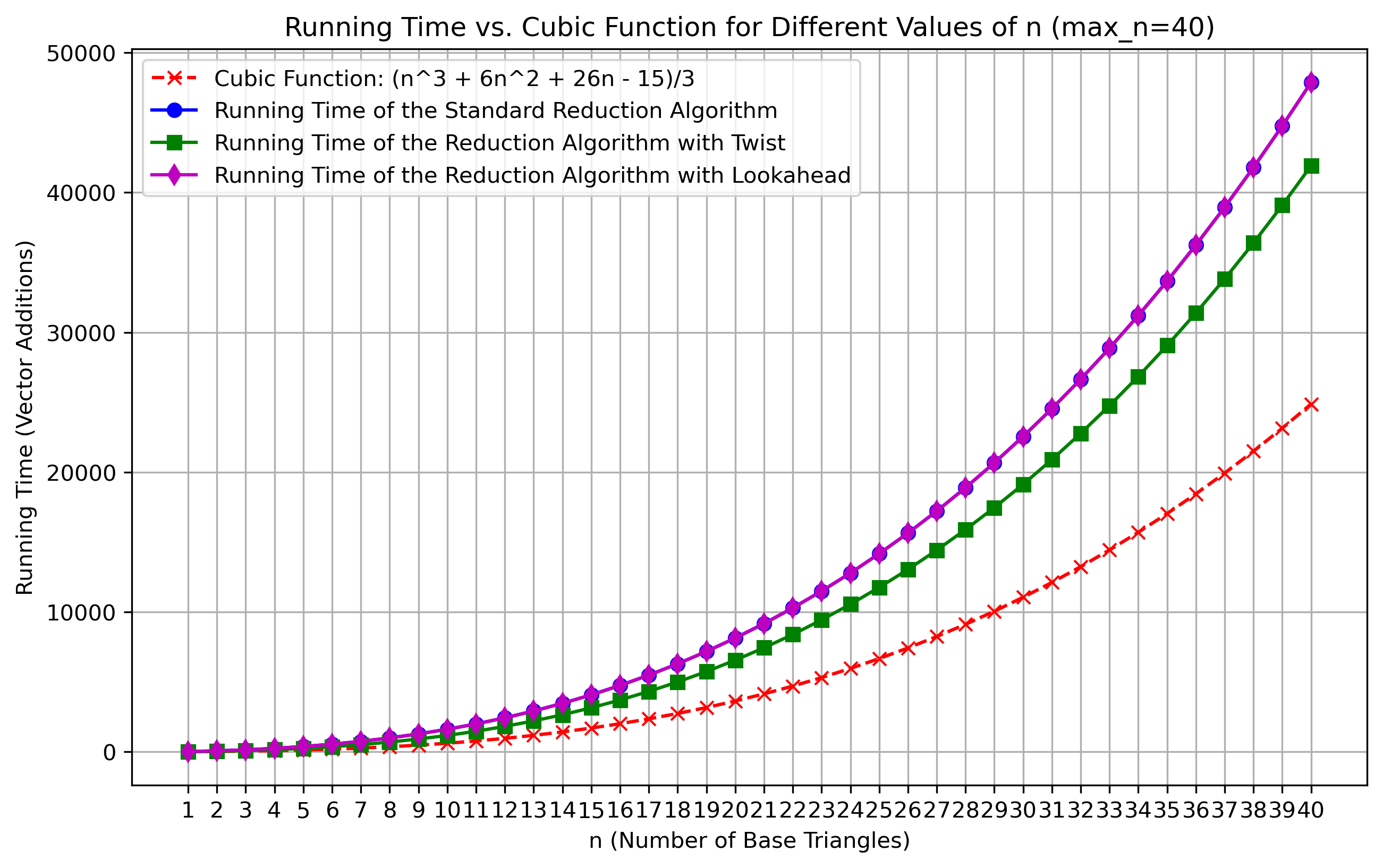}
    \caption{$\max_n=40$}
    \label{fig:runtime-n40}
  \end{subfigure}

\caption{Running times of the variants of the reduction algorithm compared with a cubic function, for two values of $\max_n$ (the maximum value of $n$ used in the experiments).}
\label{fig:runtime-two-panels}
\end{figure}

\begin{figure}[htbp]
\centering
\begin{subfigure}{0.48\textwidth}
\centering
\includegraphics[width=\textwidth]{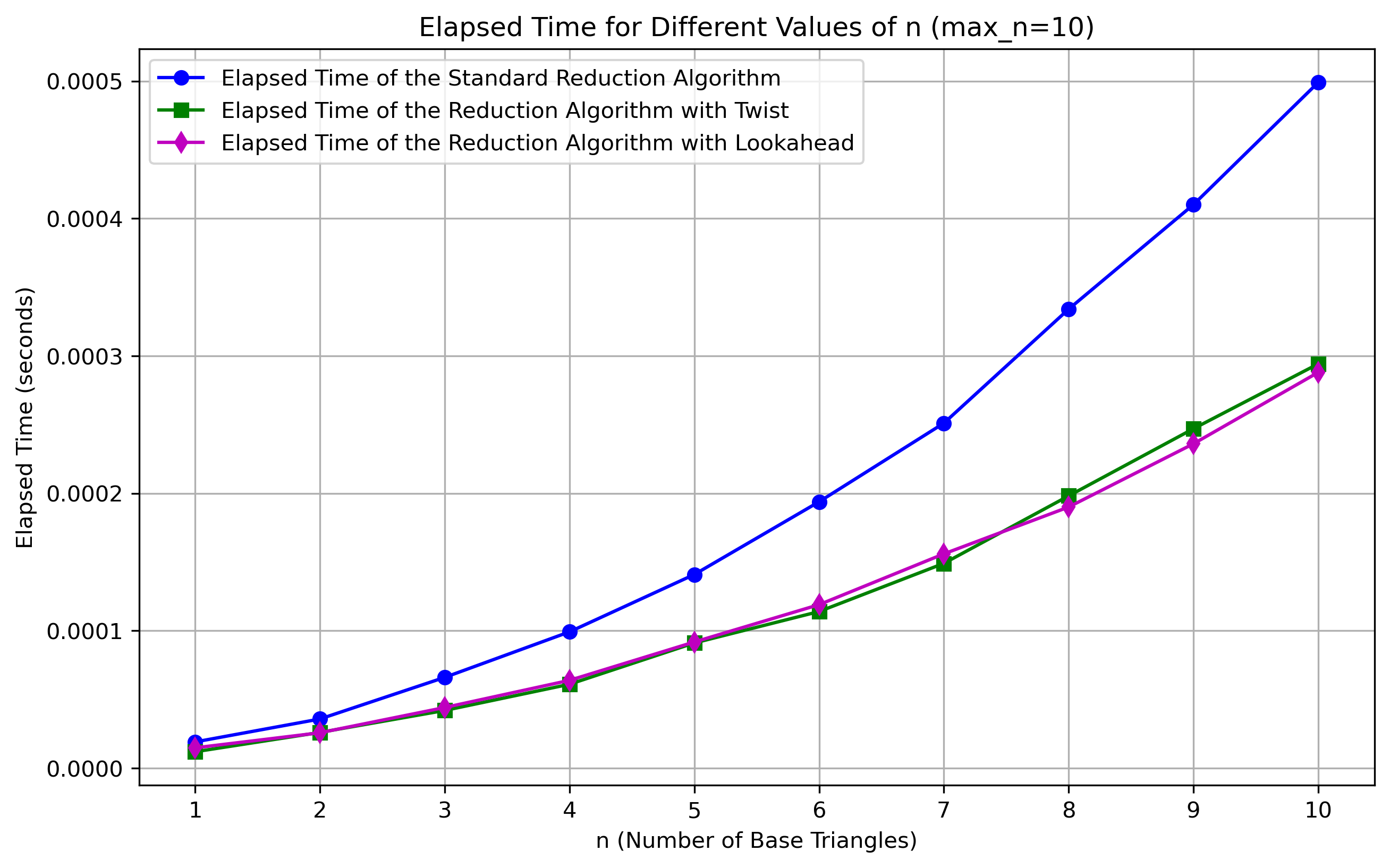}
\caption{$\max_n=10$}
\label{fig:elapsed-n10}
\end{subfigure}
\hfill
\begin{subfigure}{0.48\textwidth}
\centering
\includegraphics[width=\textwidth]{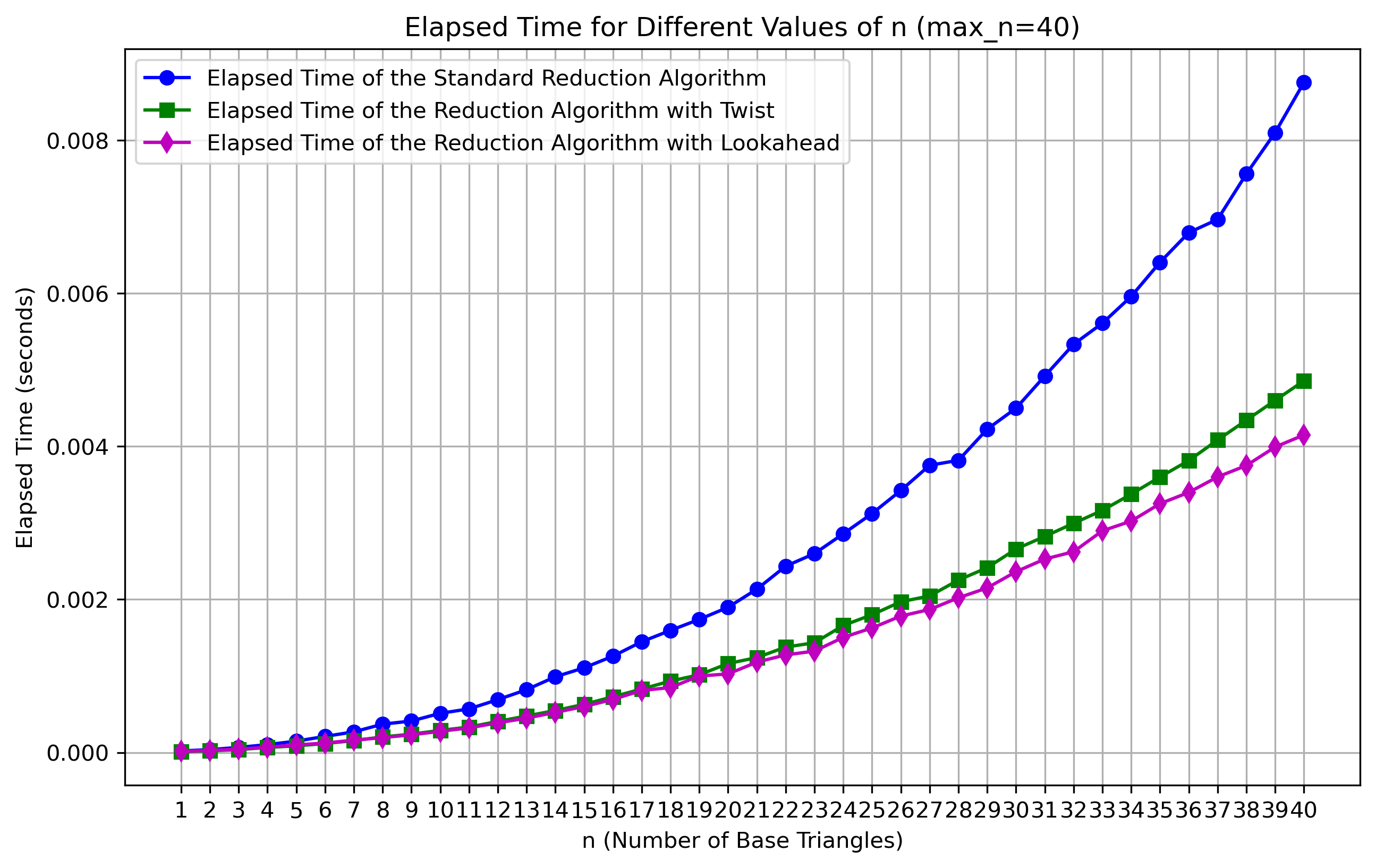}
\caption{$\max_n=40$}
\label{fig:elapsed-n40}
\end{subfigure}

\caption{Elapsed running time of different variants of the reduction algorithm as a function of the parameter $n$, for two values of $\max_n$.}
\label{fig:elapsed-two-panels}
\end{figure}

These graphs show that the running time for the standard reduction algorithm and look-ahead reduction algorithm are the same. The runtime of the reduction algorithm with twist is less than the other two algorithms. We explain the reason for this in the Example \ref{ex:ReduceB(1)}. The elapse times for these three algorithms show that both twist and look-ahead variants of the reduction algorithm perform much better than the standard algorithm when they are applied to the strip examples.   

\begin{example}\label{ex:ReduceB(1)}
  For the strip example $X(2)$, consider the matrix $B(1)$ for the boundary map $\partial _1 : C_1(X(2)) \to C_0 (X(2))$.  
\[
B(1) =\bbordermatrix{ 
& 2' & 1' & 0' & -2^* & -1^* & -2 & -1 & 2 & 1 \cr
f_1 & 0 & 0 & 0 & 0 & 1 & 0 & 1 & 0 & 0  \cr
f_2 & 0 & 0 & 0 & 1 & 0 & 1 & 0 & 0 & 0  \cr
v_1 & 1 & 0 & 0 & 1 & 1 & 0 & 0 & 1 & 1  \cr
v_2 & 0 & 1 & 1 & 0 & 0 & 0 & 1 & 0 & 1  \cr
v_3 & 0 & 1 & 0 & 0 & 0 & 1 & 0 & 1 & 0   \cr 
v_4 & 1 & 0 & 1 & 0 & 0 & 0 & 0 & 0 & 0  \cr
} 
\]     

To reduce the first five columns of the boundary matrix $B(1)$, we perform the column operations $C_1+C_3 \rightarrow C_3$ and $C_4 +C_5 \to C_5$, requiring 8 field additions. Then we carry out 9 more column additions to clear the last four columns. These operations require 36 additional field additions. If we instead apply the reduction algorithm with twist, since we first process the matrix $B(2)$, we set the last four columns of $B(1)$ to zero. This reduces the number of field additions from 44 to 8. Since the processing of $B(2)$ requires 27 field additions, the total number of field additions performed for reducing the boundary matrix decreases from $71$ to $35$. This is the drop observed for $n=2$ in the first graph in Figure \ref{fig:runtime-two-panels}. The following table shows the number of field additions performed by the variants of the reduction algorithm for $n \leq 5$.

\begin{table}[ht]
\centering
\caption{Number of field additions performed by the standard reduction algorithm and its variants.}
\label{tab:vector-additions}
\begin{tabular}{cccc}
\toprule
$n$ & Standard reduction & With twist & Look-ahead \\
\midrule
1 & 22  & 10  & 22  \\
2 & 71  & 35  & 71  \\
3 & 145 & 85  & 145 \\
4 & 247 & 155 & 247 \\
5 & 380 & 248 & 380 \\
\bottomrule
\end{tabular}
\end{table}
\end{example}

\FloatBarrier

\section{A worst-case example which is a flag complex}\label{sect:RealizingStrip}
    
In this section we show that after modifying the split complex by subdividing some edges and triangles, the resulting filtered complex is the flag complex of a filtered graph. At the end of the section we also show that these modified strip examples can be realized as the Vietoris--Rips complex of a finite data cloud for some increasing sequence of real numbers.  

\begin{definition}
Let $G=(V, E)$ be an undirected loop-free graph. The \emph{clique complex} of $G$ is a simplicial complex $Cl(G)$ whose vertex set is $V$ and a subset of vertices $\{ v_0, \dots,v_k\}$ is a simplex in $Cl(G)$ if and only if for every $i<j$, $\{v_i, v_j\}$ is an edge in $G$.
\end{definition}

Note that if $G$ is a filtered graph with filtration $\emptyset \subseteq G_1 \subseteq G_2 \subseteq \cdots \subseteq G_m$, the associated sequence of clique complexes 
$\emptyset \subseteq Cl(G_1) \subseteq Cl(G_2) \subseteq \cdots \subseteq Cl(G_m)$ gives a filtered simplicial complex. 
As defined, the split worst-case examples cannot be realized as clique complexes of filtered graphs since in the filtration the triangles appear several steps after its boundary edges. To realize them as clique complexes, we alter the split examples as follows: 
\[
\begin{tikzpicture}[scale=1.05]
\draw[] (2,0)--(0,0)--(1,1.73)--(2,0) node[midway,xshift=-24.5,yshift=-10] {$\boxed{1}$};
\draw[] (2,0)--(0,0)--(1,1.73)--(2,0) node[midway,xshift=-5,yshift=-10] {$\circled{1}$};
\draw[] (0,0)--(1,1.73)--(-1,1.73)--(0,0) node[midway,xshift=5,yshift=10] {$\boxed{2}$};
\draw[] (0,0)--(1,1.73)--(-1,1.73)--(0,0) node[midway,xshift=25,yshift=10] {$\circled{2}$};
\draw[] (4,0)--(2,0)--(3,1.73)--(4,0) node[midway,xshift=-24,yshift=-10] {$\boxed{5}$};
\draw[] (4,0)--(2,0)--(3,1.73)--(4,0) node[midway,xshift=-4.3,yshift=-10] {$\circled{5}$};
\draw[] (2,0)--(3,1.73)--(1,1.73)--(2,0) node[midway,xshift=5.8,yshift=10] {$\boxed{3}$};
\draw[] (2,0)--(3,1.73)--(1,1.73)--(2,0) node[midway,xshift=25,yshift=10] {$\circled{3}$};
\draw[] (6,0)--(8,0)--(7,1.73)--(6,0) node[midway,xshift=4.5,yshift=-12] {$\boxed{n}$};
\draw[] (6,0)--(8,0)--(7,1.73)--(6,0) node[midway,xshift=25,yshift=-12] {$\circled{n}$};
\draw[] (-4,0)--(-2,0)--(-3,1.73)--(-4,0) node[midway,xshift=5,yshift=-12,scale=0.7] {$\boxed{n-1}$};
\draw[] (-4,0)--(-2,0)--(-3,1.73)--(-4,0) node[midway,xshift=27,yshift=-12,scale=0.7] {$\circled{n-1}$};
\draw[] (6,0)--(5,1.73)--(6,1.73205)--(6,0) node[midway,xshift=-12,yshift=13, scale=0.7] {$\boxed{n-2}$};
\draw[] (6,0)--(6,1.73)--(7,1.73205)--(6,0) node[midway,xshift=-3.7,yshift=13, scale=0.7] {$\circled{n-2}$};
       
\draw[] (6,0)--(5,1.73);
\draw[] (8,0)--(7, 1.73) node[midway, yshift=5, right, scale=0.8]{$n+1=0'$};
\draw[] (5,1.73)--(6,1.73) node[midway, above, scale=0.6] {$(n-2)'$};
\draw[] (6,1.73)--(7,1.73) node[midway, above,  scale=0.6] {$(n-2)''$};
\draw[] (1,1.73)--(2,1.73) node[midway, above, scale=0.8] {$3'$};
\draw[] (2,1.73)--(3,1.73) node[midway, above, scale=0.8] {$3''$};
\draw[] (-1,1.73)--(0,1.73) node[midway, above, scale=0.8] {$2'$};
\draw[] (0,1.73)--(1,1.73) node[midway, above, scale=0.8] {$2''$};
\draw[] (0,0)--(1,0) node[midway, below, scale=0.8] {$1'$};
\draw[] (1,0)--(2,0) node[midway, below, scale=0.8] {$1''$};
\draw[] (2,0)--(3,0) node[midway, below, scale=0.8] {$5'$};
\draw[] (3,0)--(4,0) node[midway, below, scale=0.8] {$5''$};
\draw[] (6,0)--(7,0) node[midway, below,scale=0.8] {$n'$};
\draw[] (7,0)--(8,0) node[midway, below, scale=0.8] {$n''$};
\draw[] (-4,0)--(-3,0) node[midway, below,scale=0.6] {$(n-1)'$};
\draw[] (-3,0)--(-2,0) node[midway, below,scale=0.6] {$(n-1)''$};
\node[] at (2,1.73){$\bullet$};
\draw[blue] (2,0)--(2,1.73);
\node at (1,0){$\bullet$};
\draw[blue] (1,0)--(1,1.73);
\node at (0,1.73){$\bullet$};
\draw[blue] (0,0)--(0,1.73);
\node[] at (3,0){$\bullet$};
\node[] at (-3,0){$\bullet$};
\draw[blue] (3,0)--(3,1.73);
\draw[blue] (7,0)--(7,1.73);
\node[] at (7,0){$\bullet$};
\draw[blue] (-3,0)--(-3,1.73);
\draw[blue] (6,0)--(6,1.73);
            
\draw[red] (-4,0)--node[midway,yshift=-31,scale=1,xshift=-12]{$-n''$} (-4.3,4.2)-- node[midway,yshift=-36,scale=1,xshift=-10]{$-n'$} (-4.3,4.2) -- node[midway,yshift=12,xshift=6]{$-n^*$}(-3,1.73205);
\draw[red] (-3,1.73)--(-4.15,2.1);

\node[right, red, scale=0.8] at (-4.2, 2.5) {$\red{\boxed{-n}}$};
\node[right, red, scale=1] at (-4.1, 1.5) {$\red{\circled{-n}}$};

\node[left, red] at (-4.3,4.5) {$f_n$};
\node[left,red,scale=0.8,yshift=-3] at (-4.25,2.1){$g_n$};
\node[red] at (-4.15,2.1){$\bullet$};

\node[] at (-1.7,0.9){$\dots$};
\node[] at (4.7,0.9){$\dots$};

\node[] at (6,1.73){$\bullet$};
\node[above, xshift=-10] at (5,1.73){$v_{n-4}$};
\node[below, xshift=-10] at (6, 0){$v_{n-2}$};
\node[below,xshift=-10] at (-4,0){$v_{n+1}$};
\node[above] at (-3,1.73){$v_{n-1}$};
\node[below, xshift=10] at (-2, 0){$v_{n-3}$};
\node[above] at (-1,1.73) {$v_{4}$};
\node[below] at (0,0) {$v_2$};
\node[above] at (1,1.73) {$v_1$};
\node[below] at (2,0) {$v_3$};
\node[above] at (3,1.73) {$v_5$};
\node[below] at (4,0) {$v_7$};
\node[above, xshift=7] at (7,1.73) {$v_{n}$};
\node[below,xshift=10] at (8,0) {$v_{n+2}$};
\node[below,scale=0.8,blue] at (7,0){$w_n$};
\node[above,scale=0.8,blue] at (2,1.73){$w_3$};
\node[above,scale=0.8,blue] at (0,1.73){$w_2$};
\node[below,scale=0.8,blue] at (1,0){$w_1$};
\node[below,scale=0.8,blue] at (3,0){$w_5$};
\end{tikzpicture}
\]

The picture is similar to the previous example; the difference is that we add vertices $w_1, \dots, w_n$ as the midpoints of the horizontal base edges $1',\dots,n'$, and denote the resulting half edges as $i'$ and $i''$ for $i=1, \dots, n$. Similarly, we add  vertices $g_1, \dots, g_n$ at the midpoints of the edges $-1,\dots,-n$ of the fin triangles and denote the resulting half edges as $-i'$ and $-i''$ for $i=1, \dots, n$.  
    
We also add new edges that divide the base and fin triangles into half. For base triangles, these are denoted by $e^b_1,\dots,e^b_n$ and for the fin triangles, by $e^f_1,\dots,e^f_n$. They are indexed by the triangles that they subdivide. For the divided triangles, we denote the left one as boxed $i$ and the right one as the circled $i$, for base triangles. They are denoted by boxed $-i$ and circled $-i$ for fin triangles. We have the following filtration for this modified simplicial complex:
\begin{enumerate}
\item  Fin vertices $f_1,\dots, f_n$, and $g_1, \dots, g_n$,
\item  Base vertices $v_1, \dots, v_{n+2}$, and $w_1, \dots, w_n$,
\item Horizontal base edges $n'',\dots,1''$, and $n', \dots, 1'$,
\item Base edge $n+1$ which is also labeled as $0'$,
\item Fin edges that destroy fin components 
$-n^{*},\dots,-1^{*}$, 
\item Fin edges that create cycles $-n'',\dots,-1''$, and $-n', \dots, -1'$ 
\item Vertical base edges $n, n-1, \dots, 1$, 
\item Base midlines and base triangles $e^b _1, \boxed{1}, \circled{1},\dots, e^b_n , \boxed{n}, \circled{n}$, 
\item Fin midlines and fin triangles $e^f _1, \boxed{-1}, \circled{-1}, \dots,  e^f_n, \boxed{-n}, \circled{-n}$.
\end{enumerate}

Note that for all $i$, the edges $e^b _i$ and the base triangles $\boxed{i}$ and ${\circled{i}}$ appear at the same time in the filtration. Similarly for the fin triangles,
the edges $e^f _i$ and the fin triangles $\red{\boxed{-i}}$ and $\red{\circled{i}}$ appear at the same filtration level. We call this modified example the \emph{modified split example} and denote it by $Y(n)$. 

For each $k=0,1,2$, let $B'(k)$ denote the boundary matrix for $\partial _k: C_k (Y(n) )\to C_{k-1} (Y(n))$.  For the boundary matrix, we use the linear ordering as given in the above list.

\begin{example}\label{ex:B'(2)}
In this example we illustrate how the boundary matrix $B'(2)$ for the modified split example $Y(2)$ is reduced using the reduction algorithm. We explain the processing of the columns  using the sparse matrix notation:
\begin{align*}
\boxed{1} & \to (e_1 ^b, 1, 1') \\
d_1= \boxed{1}+ \circled{1} & \to (e_1 ^b, 1, 1')
+(e_1^b, 2, 1'') \to (1, 2, 1', 1'')
\end{align*}
So the triangle $\boxed {1}$ is paired with $e_1^b$, and the triangle $\circled{1}$ is paired with edge $1$.  For the triangles $\boxed{2}$ and $\circled{2}$, we have 
\begin{align*} \boxed{2} & \to (e_2^b, 0', 2') \\
\boxed{2} + \circled{2} & \to (e_2 ^b, 0', 2')+ 
(e_2^b, 1, 2'') \to (1, 0', 2', 2'') \\
d_2= (\boxed{2}+ \circled{2})+d_1& \to (1, 0', 2', 2'')
+ (1, 2,1', 1'') =(2, 0', 1', 1'', 2', 2'').
\end{align*}
This shows that the triangle $\boxed{2}$ is paired with $e_2 ^b$ and the triangle $\circled{2}$ is paired with $2$. For the fin triangles we have a similar pairing process. 
\begin{align*}
{\red \boxed{-1}} & \to (e_1^f, -1', -1^*) \\
{\red \boxed{-1}} + {\red \circled{-1}} & \to (e_1^f, -1', -1^*) + (e_1^f, 1, -1'') \to (1, -1', -1'', -1^*)  \\
({\red \boxed{-1}} + {\red \circled{-1}}) +d_1 &\to(1, -1', -1'', -1^*)  + (1, 2, 1', 1'') \to (2, -1', -1'', -1^*, 1', 1'') \\
({\red \boxed{-1}} + {\red \circled{-1}} +d_1)+d_2 & \to(2, -1', -1'', -1^*, 1', 1'') +(2, 0', 1', 1'', 2', 2'') \\
& \to (-1', -1'', -1^*, 0', 2', 2'').
\end{align*}
This pairs the fin triangle $\red \boxed{-1}$ with $e_1^f$ and the fin triangle $\red \circled{-1}$ with the edge $-1'$. Finally we have
\begin{align*}
{\red \boxed{-2}} & \to (e_2^f, -2', -2^*) \\
{\red \boxed{-2}}+ {\red \circled{-2}} & \to 
(e_2^f, -2', -2^*) + (e_2^f, 2, -2'') \to (2, -2', -2'', -2^* ) \\    
({\red \boxed{-2}}+{\red \circled{-2}})+ d_2 & \to (2, -2', -2'',  -2^*)+ (2, 0', 1', 1'', 2', 2'') \\ 
& \to (-2',-2'', -2^*, 0', 1', 1'', 2', 2'').  
\end{align*}
This pairs the triangle $\red \boxed{-2}$ with $e_2^f$ and the triangle $\red \circled{-2}$ with the edge $-2'$.
So the persistent pairs for $Y(2)$ for the 1-dimensional homology $H_1$ are
\[
\{ (e_1 ^b, \boxed{1}), (1, \circled{1}), (e_2 ^b , \boxed{2}), (2, \circled{2}), (e_1^f, {\red \boxed{-1}} ), (-1', {\red \circled{-1}}), (e_2 ^f, {\red \boxed{-2}}),  (-2', {\red \circled{-2}} ) \}.
\]
Since some of the pairs consist of simplices that appear at the same filtration level, they will appear as points on the diagonal line in the persistent diagram.

\end{example}

Now we are ready to prove the following:

\begin{proposition}\label{pro:Modified}
The modified strip example $Y(n)$ is the clique complex of a filtered graph and it has time complexity $\Omega({M^3})$ where $M$ is the number of simplices. 
\end{proposition}
    
\begin{proof} At each filtration level the modified  strip example is the clique complex of its one skeleton because all the triangles in the complex appear at the same filtration level as the boundary edges of the triangles. 

To see that the reduction algorithm applied to $B'(n)$ has cubic runtime complexity, observe that the complex $Y(n)$ has 
$4n+2$ vertices, $8n+1$ edges, and $4n$ triangles. So the number of simplices in $Y(n)$ is $M=16n+3$.
Note that in the reduction algorithm, the boxed triangles are always paired with the edge that divides the triangle into two. To obtain the chains $t_i=\boxed {i} +\circled {i}$ and $t_{-i}= {\red \boxed{-i}}+{\red \circled{-i}}$ we perform $2n$ vector additions which is performed using $12n$ field additions. 

The processing of the chains $t_i$ and $t_{-i}$ are very similar to the processing of the triangles in the original strip example. In fact we use the same number of vector additions to process these chains. The only difference is that the vectors added have more nonzero entries since more edges are introduced in the modified case. This increases the number of field additions in the reduction process. We now calculate exact number of fields additions used to process columns for the boundary matrix $B'(n)$.   

Let $d_i =t_1+\cdots+t_i$. Note that the columns for $d_i$ have $2i+2$ nonzero entries. Since the column for $d_i$ is obtained by adding the columns $t_i$ and $d_{i-1}$, the number of field additions to process all the columns for $t_1, \dots, t_n$ is
\[
(4+4)+(6+4)+(8+4)+\cdots+(2n+4)=n^2+5n-6.
\]

To process the column for $t_{-j}$, we perform the additions $t_{-j} +d_j + \cdots+ d_n$ which requires 
\[
4 + \sum _{i=j} ^n (2i+2)=n^2 +3n-j^2-j+6
\]
field additions. Adding these over all $j=1, \dots, n$, we see that to process the columns for fin triangles, we perform 
$ \frac{2}{3} n (n^2+3n+8)$ additions. Hence the total number of field additions to reduce the matrix $B'(2)$ is 
\[
S'(n)= \frac{1}{3} ( 2n^3+9n^2+37n-18).
\]
Since $M=16n+3$, we have $S'(n)=\Omega (M)$. Note that both the number of simplices and the number of field additions in the reduction process increases roughly by a factor of $2$ for the modified complex $Y(n)$ compared to the reduction process for the strip example $X(n)$. 
\end{proof}

Next we show that the modified strip examples can be realized as the Vietoris--Rips complex of a finite data cloud for a sequence of real number $0<r_1< \cdots < r_s$ as scale parameters.
We first recall the definition of a Vietoris--Rips complex of a data cloud.

\begin{definition} Let $X$ be a finite set of points in $\mathbb R ^l$. For a real number $r\geq 0$, the \emph{Vietoris--Rips complex} $VR (X, r)$ is defined as the simplicial complex whose vertex set is $X$, and whose $n$-simplices are the subsets $\{ x_0, \dots , x_n\}$ of $X$ such that $d( x_i, x_j ) \leq r$ for every $i, j$.
\end{definition}

For an increasing sequence $0 < r_1 < r_2 < \cdots < r_m$ of positive real numbers, we can define a filtered simplicial complex by taking the simplicial complex $K_i$ to be the Vietoris--Rips complex $VR(X, r_i)$.  The following proposition is implicit in classical distance geometry and the theory of Vietoris--Rips complexes, but  we were unable to find a direct reference. We include a proof here for completeness.

\begin{proposition}\label{pro:Realization}
Let $\emptyset \subseteq G_1 \subseteq \cdots \subseteq G_m
$ be a filtered graph on a finite vertex set $V=\{1,\dots,n\}$, and let $\mathrm{Cl}(G_t)$ denote the clique
complex of $G_t$ for $t=1,\dots,m$. Then there exist points
\[
X=\{p_1,\dots,p_n\}\subset \mathbb{R}^l
\] 
and increasing sequence of real numbers $0<r_1<\cdots<r_m$ such that for each $t=1,\dots, m$, there is an isomorphism of simplicial complexes
\[
VR(X,r_t) \cong \mathrm{Cl}(G_t).
\]
\end{proposition}

\begin{proof}
It is well-known that Vietoris--Rips complexes are flag complexes, hence they are clique complexes of underlying 1-dimensional subcomplexes (see for example \cite[Section~III.2]{EdelsbrunnerHarer-2010}). Therefore it suffices to realize the
filtered graphs $G_i$ as nested distance--threshold graphs.

Since the $n=1$ case is trivial, we can assume $n\geq 2$. For each unordered pair $\{i,j\}$, define its \emph{appearance level}
\[
\ell(i,j)=\min\{t : \{i,j\}\in E(G_t)\},
\]
with $\ell(i,j)=\infty$ if the edge never appears. Note that $\ell (i, j) \geq 1$ for all $i,j$. Fix
\[
\varepsilon := \frac{1}{n-1},
\qquad
\delta_t := 2^{-t}\varepsilon
\]
for $t=1,\dots,m$, and let $\delta_{m+1} := 0$.
Then we have $\frac12 \geq \delta_1>\cdots>\delta_m>\delta_{m+1}=0$ and for each $i$,
\[
\sum_{j\neq i} \delta_{\ell(i,j)}
\le
(n-1)\delta_1
= (n-1) \frac{\epsilon}{2} =
\frac12.
\]

Define an $n\times n$ symmetric matrix $K=(K_{ij})$ by
\[
K_{ii}=1, \qquad K_{ij}=\delta_{\ell(i,j)} \ \text{ for all }\ i\neq j.
\]
By Gershgorin's circle theorem (see, for example, \cite{Shores-2018}), every eigenvalue of $K$ lies in the interval
$[1-\frac12,\,1+\frac12]$. In particular $K$ is positive definite.

By a standard result in linear algebra, every real symmetric positive definite matrix is a Gram matrix.
Thus there exist vectors $p_1,\dots,p_n\in\mathbb{R}^l$ such that
\[
K_{ij}=\langle p_i,p_j\rangle
\]
(see \cite[Theorem~7.35]{Axler-2015}). In particular, $\|p_i\|=1$ for all $i$, and
\[
\|p_i-p_j\|^2
= \|p_i\|^2+\|p_j\|^2-2\langle p_i,p_j\rangle
= 2-2\delta_{\ell(i,j)}.
\]

For each $t=1,\dots, m+1$, define 
\[
d_t := \sqrt{2-2\delta_t}.
\]
For all $i,j$, we have $\| p_i-p_j \| =d_{\ell(i,j)}$.
Since the sequence $\delta_t$ is strictly decreasing, the sequence $d_t$ is strictly increasing. We have
$$ 1 \leq d_1 < d_2 < \cdots < d_m < d_{m+1}=\sqrt{2}.$$
For each $t=1,\dots,m$, choose a radius $r_t$ such that
\[
d_t < r_t < d_{t+1}.
\]
Then
\[
\|p_i-p_j\|\le r_t
\quad\Longleftrightarrow\quad
d_{\ell (i,j)} \leq d_t 
\quad\Longleftrightarrow\quad
\ell(i,j)\le t
\quad\Longleftrightarrow\quad
\{i,j\}\in E(G_t).
\]
Hence the distance--threshold graph at scale $r_t$ is exactly $G_t$. This completes the proof.
\end{proof}

We finish the paper with the following corollary.

\begin{corollary}
The modified strip example $Y(n)$ can be realized as the Vietoris--Rips complex of a finite data set for some sequence of real numbers $0<r_1<\cdots< r_m$ as scale parameters.
\end{corollary}

\begin{proof} This follows from Propositions \ref{pro:Modified} and \ref{pro:Realization}.
\end{proof}

Worst-Case examples for Erd\" os-R\' enyi and Vietoris-Rips filtration models were also given in \cite{GHK-2022}. For these models the worst-case runtime is $\Theta (n^7)$ since the processing of the all higher dimensional simplices is also considered.


\section*{Acknowledgements and Funding Declaration}
The second author is supported by T\" UB\. ITAK 2219-International Postdoctoral Research Fellowship Program (2023, 2nd term). We gratefully acknowledge T\" UB\. ITAK for its support of this research. We thank Matthew Gelvin, Bar\i \c s Co\c skun\" uzer, Caroline Yal\c c\i n, and Koray Karabina for reading an earlier version of the paper and for their helpful comments. Part of the paper was written while the second author was on sabbatical leave visiting University of Waterloo. The second author thanks University of Waterloo for hospitality during his sabbatical leave.


\bibliographystyle{plain}
\bibliography{bibliography-Final}


\appendix

\section{An algorithm for constructing the strip examples}\label{sect:Algorithm}

An algorithm to construct the filtered simplicial complex strip $X(n)$ for $n\geq 2$ can be given as follows. The complex $X(1)$ can be constructed easily by inserting the list of simplices from the picture.

\begin{algorithm}[htbp]
\caption{Building ordered lists of vertices, edges, and triangles of the simplicial complex $X(n)$ for $n\geq 2$}
\label{alg:Building}
\begin{algorithmic}[1]

\Function{BuildVertices}{$n$}
    \State \Return $[f_1,\dots,f_n,v_1,\dots,v_{n+2}]$
\EndFunction

\Function{BuildEdges}{$n$}
    \State $E \gets [\,]$

    \Comment{Horizontal base edges $n',\dots,1', 0'$}
    \For{$i=n,n-1,\dots,3$}
        \State append $(v_{i-2},v_{i+2})$ to $E$ \Comment{edge $i'$}
    \EndFor
    \State append $(v_1,v_4)$ to $E$ \Comment{edge $2'$}
    \State append $(v_2,v_3)$ to $E$ \Comment{edge $1'$}
    \State append $(v_n,v_{n+2})$ to $E$ \Comment{edge $0'$}

    \Comment{Fin edges $-n^*,\dots,-1^*$}
    \For{$i=n,n-1,\dots,2$}
        \State append $(f_i,v_{i-1})$ to $E$ \Comment{edge $-i^*$}
    \EndFor
    \State append $(f_1,v_1)$ to $E$ \Comment{edge $-1^*$}

    \Comment{Fin edges $-n,\dots,-1$}
    \For{$i=n,n-1,\dots,1$}
        \State append $(f_i,v_{i+1})$ to $E$ \Comment{edge $-i$}
    \EndFor

    \Comment{Vertical base edges $n,\dots,1$}
    \For{$i=n,n-1,\dots,2$}
        \State append $(v_{i-1},v_{i+1})$ to $E$ \Comment{edge $i$}
    \EndFor
    \State append $(v_1,v_2)$ to $E$ \Comment{edge $1$}

    \State \Return $E$
\EndFunction

\Function{BuildTriangles}{$n$}
    \State $T \gets [\,]$

    \Comment{Base triangles $\circled{1},\circled{2},\dots,\circled{n}$}
    \State append $(v_1,v_2,v_3)$ to $T$ \Comment{triangle $\circled{1}$}
    \State append $(v_1,v_2,v_4)$ to $T$ \Comment{triangle $\circled{2}$}
    \For{$i=3,\dots,n$}
        \State append $(v_{i-2},v_i,v_{i+2})$ to $T$ \Comment{triangle $\circled{i}$}
    \EndFor

    \Comment{Fin triangles ${\red \circled{-1}},{\red \circled{-2}},{\red \dots,\circled{-n}}$}
    \State append $(f_1,v_1,v_2)$ to $T$ \Comment{triangle ${\red \circled{-1}}$}
    \For{$i=2,\dots,n$}
        \State append $(f_i,v_{i-1},v_{i+1})$ to $T$ \Comment{triangle ${\red \circled{-i}}$}
    \EndFor

    \State \Return $T$
\EndFunction

\end{algorithmic}
\end{algorithm}

\begin{algorithm}[htbp]
\caption{Constructing the filtered simplicial complex $X(n)$ for $n\geq 2$}
\label{alg:Constructing}
\begin{algorithmic}[1]

\Function{CreateWorstCaseComplex}{$n$}
    \State $V \gets$ \Call{BuildVertices}{$n$}
    \State $E \gets$ \Call{BuildEdges}{$n$}
    \State $T \gets$ \Call{BuildTriangles}{$n$}
    \State $\mathcal{K} \gets [\,]$
    \State $t \gets 1$

    \For{each $v\in V$}
        \State append $(\{v\},t)$ to $\mathcal{K}$
        \State $t \gets t+1$
    \EndFor

    \For{each $e\in E$}
        \State append $(\mathrm{sort}(e),t)$ to $\mathcal{K}$
        \State $t \gets t+1$
    \EndFor

    \For{each $\sigma\in T$}
        \State append $(\mathrm{sort}(\sigma),t)$ to $\mathcal{K}$
        \State $t \gets t+1$
    \EndFor

    \State \Return $\mathcal{K}$
\EndFunction

\end{algorithmic}
\end{algorithm}

\FloatBarrier

\end{document}